\documentclass[final]{siamltex}

\usepackage{graphicx}
\usepackage{color}
\usepackage{bm}
\usepackage{amsfonts}
\usepackage{array}
\usepackage{amsmath}

\newcommand{\tr}{^{\sf T}}
\newcommand{\m}[1]{{\bf{#1}}}
\newcommand{\g}[1]{\bm #1}
\newcommand{\C}[1]{{\cal {#1}}}
\renewcommand{\bar}{\overline}
\newtheorem{remark}{Remark}[section]

\title
{Convergence rate for a Gauss collocation method applied to
constrained optimal control
\thanks{
July 10, 2016.
Revised December 16, 2017.
The authors gratefully acknowledge support by
the Office of Naval Research under grant
N00014-15-1-2048, by the National Science Foundation under
grant DMS-1522629, and by the U.S. Air Force Research Laboratory
under contract FA8651-08-D-0108/0054.}
}
\author{
William W. Hager\thanks{{\tt hager@ufl.edu},
        http://people.clas.ufl.edu/hager/,
        PO Box 118105,
        Department of Mathematics,
        University of Florida, Gainesville, FL 32611-8105.
        Phone (352) 294-2308. Fax (352) 392-8357.}
\and
Jun Liu\thanks{{\tt juliu@siue.edu},
        http://www.siue.edu/$\sim$juliu/,
        Department of Mathematics and Statistics,
        Southern Illinois University Edwardsville, Edwardsville, IL 62026.
        Phone (618) 650-2220.}
\and
Subhashree Mohapatra\thanks{{\tt subha@ufl.edu},
        Department of Mathematics,
        University of Florida, Gainesville, FL 32611.}                    
\and
Anil V. Rao\thanks{{\tt anilvrao@ufl.edu},
        http://www.mae.ufl.edu/rao,
        Department of Mechanical and Aerospace Engineering,
        P.O. Box 116250, Gainesville, FL 32611-6250.
        Phone (352) 392-0961. Fax (352) 392-7303.}
\and
Xiang-Sheng Wang\thanks{{\tt xswang@louisiana.edu},
        http://www.ucs.louisiana.edu/$\sim$xxw6637/,
        Department of Mathematics,
        University of Louisiana at Lafayette,
        Lafayette, LA 70503.
        Phone (337) 482-5281.}
}
\begin{document}
\maketitle
\begin{abstract}
A local convergence rate is established for
a Gauss orthogonal collocation method applied to optimal control
problems with control constraints.
If the Hamiltonian possesses a strong convexity property, then
the theory yields convergence for problems whose optimal state and
costate possess two square integrable derivatives.
The convergence theory is based on a stability result for the sup-norm
change in the solution of a variational inequality relative to a 2-norm
perturbation,
and on a Sobolev space bound for the error in interpolation at the
Gauss quadrature points and the additional point $-1$.
The tightness of the convergence theory is examined using a numerical example.
\end{abstract}


\pagestyle{myheadings} \thispagestyle{plain}
\markboth{W. W. HAGER, J. LIU, S. MOHAPATRA, A. V. RAO, and X.-S. WANG}
{GAUSS COLLOCATION IN CONSTRAINED OPTIMAL CONTROL}

\section{Introduction}
In earlier work \cite{HagerHouRao16a,HagerHouRao15c,HagerHouRao16c},
we analyze the convergence rate for orthogonal collocation methods
applied to unconstrained control problems.
In this analysis, it is assumed that the problem solution is smooth,
in which case the theory implies that the discrete approximations
converge to the solution of the continuous problem at potentially
an exponential rate.
But when control constraints are present, the solution often
possesses limited regularity.
The convergence theory developed in the earlier work for unconstrained problems
required that the optimal state had at least four derivatives, while for
constrained problems, the optimal state may have only two derivatives,
at best \cite{Bonnans10, DontchevHager98a, Hag2, Hermant09}.
The earlier convergence theory was based on a stability analysis
for a linearization of the unconstrained control problem;
the theory showed that the
sup-norm change in the solution was bounded relative to the
sup-norm perturbation in the linear system.
Here we introduce a convex control constraint, in which case
the linearized problem is a variational inequality,
or equivalently a differential inclusion, not a linear system.
We obtain a bound for the sup-norm change in the solution
relative to a 2-norm perturbation in the variational inequality.
By using the 2-norm for the perturbation rather than the sup-norm,
we are able to avoid both Lebesgue constants and the Markov bound
\cite{Markov1916} for the sup-norm of the derivative of a polynomial
relative to the sup-norm of the original polynomial.
Using best approximation results in Sobolev spaces
\cite{BernardiMaday92, Elschner93}, we obtain convergence
when the optimal state and costate have only two square integrable derivatives,
which implies that the theory is applicable to a class of control
constrained problems for which the optimal control is Lipschitz continuous.

The specific collocation scheme analyzed in this paper,
presented in \cite{Benson2,GargHagerRao10a},
is based on collocation at the Gauss quadrature points,
or equivalently, at the roots of a Legendre polynomial.
Other sets of collocation points that have been studied in the literature
include the Lobatto quadrature points
\cite{Elnagar1,Fahroo2, GongRossKangFahroo08},
the Chebyshev quadrature points \cite{Elnagar4,FahrooRoss02},
the Radau quadrature points
\cite{Fahroo3,GargHagerRao11a,LiuHagerRao15,PattersonHagerRao14},
and extrema of Jacobi polynomials \cite{Williams1}.
Kang \cite{kang08,kang10} obtains a convergence rate for the Lobatto
scheme applied to control systems in feedback linearizable normal form
by inserting bounds in the discrete problem for the states, the controls,
and certain Legendre polynomial expansion coefficients.
In our approach, the discretized problem is obtained by simply
collocating at the Gauss quadrature points.

Our approximation to the control problem uses a global polynomial defined
on the problem domain.
Earlier work, including
\cite{DontchevHager93,DontchevHager97,DontchevHagerMalanowski00,
DontchevHagerVeliov00,Hager99c,Kameswaran1,Reddien79},
utilizes a piecewise polynomial approximation, in which case convergence
is achieved by letting the mesh spacing approach zero, while keeping
the polynomial degree fixed.
For an orthogonal collocation scheme based on global polynomials,
convergence is achieved by letting the degree of the polynomials
tend to infinity.
Our results show that even when control constraints are present, and
a solution possesses limited regularity, convergence can still be achieved
with global polynomials.

We consider control problems of the form
\begin{equation}\label{P}
\begin{array}{cll}
\mbox {minimize} &C(\m{x}(1))&\\
\mbox {subject to} &\dot{\m{x}}(t)=
\m{f(x}(t), \m{u}(t)), \quad \m{u}(t) \in \C{U}, \quad t\in\Omega,\\
&\m{x}(-1)=\m{x}_0, \quad
(\m{x}, \m{u}) \in \C{C}^1(\Omega; \;\mathbb{R}^n) \times
\C{C}^0 (\Omega; \;\mathbb{R}^m),
\end{array}
\end{equation}
where $\Omega = [-1, 1]$, the control constraint set
$\C{U} \subset \mathbb{R}^m$
is closed and convex with nonempty interior,
the state ${\m x}(t)\in \mathbb{R}^n$,
$\dot{\m x}$ denotes the derivative of $\m{x}$ with respect to $t$,
${\m x}_0$ is the initial condition which we assume is given,
${\m f}: {\mathbb R}^n \times {\mathbb R}^m\rightarrow {\mathbb R}^n$,
$C: {\mathbb R}^n \rightarrow {\mathbb R}$,
$\C{C}^l(\Omega;\; \mathbb{R}^n)$
denotes the space of $l$ times continuously differentiable functions
mapping $\Omega$ to $\mathbb{R}^n$.
It is assumed that $\m{f}$ and $C$ are at least continuous.

Let $\C{P}_N$ denote the space of polynomials of degree at most $N$,
and let $\C{P}_N^n$ denote the
$n$-fold Cartesian product $\C{P}_N \times \ldots \times \C{P}_N$.
We analyze the discretization of (\ref{P}) given by
\begin{equation}\label{D}
\begin{array}{cl}
\mbox {minimize } &C(\m{x}(1))\\
\mbox {subject to } &\dot{\m{x}}(\tau_i)=
\m{f(x}(\tau_i), \m{u}_i),\quad \m{u}_i \in \C{U}, \quad 1 \le i \le N,\\
&\m{x}(-1)=\m{x}_0, \quad \m{x} \in \C{P}_N^n.
\end{array}
\end{equation}
The polynomials used to approximate the state should satisfy the
dynamics exactly at the collocation points $\tau_i$, $1 \le i \le N$.
The parameter $\m{u}_i$ represents an approximation to the control
at time $\tau_i$.
The dimension of $\C{P}_N$ is $N+1$, while there are
$N+1$ equations in (\ref{D}) corresponding to the
collocated dynamics at $N$ points and the initial condition.
We collocate at the Gauss quadrature points,
which are symmetric about $t = 0$ and which satisfy
\[
-1 < \tau_1 < \tau_2 < \ldots < \tau_N < +1 .
\]
The analysis also makes use of the two noncollocated points
\[
\tau_0 = -1 \quad \mbox{and} \quad \tau_{N+1} = +1.
\]

For $\m{x} \in \C{C}^0(\Omega;\; \mathbb{R}^n)$,
we use the sup-norm $\| \cdot \|_\infty$ given by
\[
\|\m{x}\|_\infty = \sup \{ |\m{x} (t)| : t \in [0, 1] \} ,
\]
where $| \cdot |$ is the Euclidean norm.
Given $\m{y} \in \mathbb{R}^n$, the ball with center $\m{y}$ and radius
$\rho$ is denoted
\[
\C{B}_\rho (\m{y}) := \{ \m{x} \in \mathbb{R}^n :
|\m{x} - \m{y}| \le \rho \} .
\]
The following regularity assumption is assumed to hold throughout the paper.

{\bf Smoothness.}
The problem (\ref{P}) has a local minimizer
$(\m{x}^*, \m{u}^*)$ in
$\C{C}^1 (\Omega;\; \mathbb{R}^n) \times \C{C}^0 (\Omega;\; \mathbb{R}^m)$.
There exists an open set
$\C{O} \subset \mathbb{R}^{m+n}$ and $\rho > 0$ such that
\[
\C{B}_\rho (\m{x}^*(t),\m{u}^*(t)) \subset \C{O} \mbox{ for all }
t \in \Omega.
\]
Moreover, the first two derivatives of $f$ and $C$ are
Lipschitz continuous on the closure of
$\C{O}$ and on $\C{B}_\rho (\m{x}^*(1))$ respectively.

Let $\g{\lambda}^*$ denote the solution of the
linear costate equation
\begin{equation}\label{costate}
\dot{\g{\lambda}}^*(t)=-\nabla_xH({\m x}^*(t), {\m u}^*(t), {\g \lambda}^*(t)),
\quad {\g \lambda}^*(1)=\nabla C({\m x}^*(1)),
\end{equation}
where $H$ is the Hamiltonian defined by
$H({\m x}, {\m u}, {\g \lambda}) ={\g\lambda}\tr {\m f}({\m x}, {\m u})$ and
$\nabla$ denotes gradient.
From the first-order optimality conditions (Pontryagin's minimum principle),
it follows that
\begin{equation} \label{controlmin}
-\nabla_u H({\m x}^*(t), {\m u}^*(t), {\g \lambda}^*(t)) \in
N_\C{U} (\m{u}^*(t)) \quad \mbox{for all } t \in \Omega,
\end{equation}
where $N_\C{U}$ is the normal cone.
For any $\m{u} \in \C{U}$,
\[
N_\C{U}(\m{u}) = \{ \m{w} \in \mathbb{R}^m :
\m{w}\tr(\m{v} - \m{u}) \le 0 \mbox{ for all } \m{v} \in \C{U} \},
\]
while $N_\C{U}(\m{u}) = \emptyset$ if $\m{u} \not\in \C{U}$.

Since the collocation problem (\ref{D}) is finite dimensional,
the first-order optimality conditions, or Karush-Kuhn-Tucker conditions,
hold when a constraint qualification
\cite{NocedalWright2006} is satisfied.
We show in Lemma~\ref{equiv} that the first-order optimality conditions
are equivalent to the existence of $\g{\lambda} \in \C{P}_N^n$ such that
\begin{eqnarray}
\dot{\g \lambda}(\tau_i) &=&
-\nabla_x H\left({\m x} (\tau_i),{\m u}_i, {\g \lambda} (\tau_i) \right),
\quad 1 \leq i \leq N,  \label{dcostate} \\
{\g \lambda}(1) &=& \nabla C(\m{x}(1)), \label{dterminal}\\
N_\C{U}(\m{u}_i) &\ni&
-\nabla_u H\left({\m x}(\tau_i),{\m u}_i, {\g \lambda} (\tau_i) \right),
\quad 1\leq i\leq N. \label{dcontrolmin}
\end{eqnarray}

The following assumptions are utilized in the convergence analysis.
\smallskip

\begin{itemize}
\item[(A1)]
For some $\alpha > 0$,
the smallest eigenvalue of the Hessian matrices $\nabla^2 C(\m{x}^*(1))$
and $ \nabla^2_{(x,u)} H(\m{x}^* (t), \m{u}^* (t), \g{\lambda}^* (t) )$
are greater than $\alpha$, uniformly for $t \in [0, 1]$.
\item[(A2)]
For some $\beta < 1/2$,
the Jacobian of the dynamics satisfies
\[
\|\nabla_x \m{f} (\m{x}^*(t), \m{u}^*(t))\|_\infty \le \beta
\quad \mbox{and} \quad
\|\nabla_x \m{f} (\m{x}^*(t), \m{u}^*(t))\tr\|_\infty \le \beta
\]
for all $t \in \Omega$ where $\| \cdot \|_\infty$ is the matrix
sup-norm (largest absolute row sum), and the Jacobian
$\nabla_x \m{f}$ is an $n$ by $n$  matrix whose $i$-th row is
$(\nabla_x f_i)\tr$.
\end{itemize}
\smallskip

The condition (A2) ensures (see Lemma~\ref{feasible})
that in the discrete linearized problem, it is possible to solve for the
discrete state in terms of the discrete control.
As shown in \cite{HagerHouRao16a}, this property holds
in an $hp$-collocation framework when the domain $\Omega$
is partitioned into $K$ mesh intervals with $K$
large enough that
\[
\|\nabla_x \m{f} (\m{x}^*(t), \m{u}^*(t))\|_\infty/K \le \beta
\quad \mbox{and} \quad
\|\nabla_x \m{f} (\m{x}^*(t), \m{u}^*(t))\tr\|_\infty/K \le \beta
\]
for all $t \in \Omega$.

The coercivity assumption (A1) is not only a sufficient condition for
the local optimality of a feasible point $(\m{x}^*, \m{u}^*)$ of (\ref{P}),
but it yields the stability of the discrete linearized problem
(see Lemma~\ref{inf-bounds}).
One would hope that (A1) could be weakened to only require coercivity
relative to a subspace associated with the linearized dynamics similar to
what is done in \cite{DontchevHager93}.
To formulate this weakened condition, we introduce the following 6 matrices:
\[
\begin{array}{ll}
{\bf A}(t)=\nabla_x{\bf f}({\bf x}^*(t),{\bf u}^*(t)),
&{\bf B}(t)=\nabla_u{\bf f}({\bf x}^*(t),{\bf u}^*(t)),\\[.1in]
{\bf Q}(t)=\nabla_{xx}H\left({\bf x}^*(t),{\bf u}^*(t),
{\g \lambda}^*(\tau_i)\right),
&{\bf S}(t)=\nabla_{ux}H\left({\bf x}^*(t),{\bf u}^*(t),
{\g \lambda}^*(\tau_i)\right),\\[.1in]
{\bf R}(t)=\nabla_{uu}H\left({\bf x}^*(t),{\bf u}^*(t),
{\g \lambda}^*(\tau_i)\right), &{\bf T}=\nabla^2C({\bf x}^*(1)).
\end{array}
\]
With this notation and with
$\langle \cdot , \cdot \rangle$ denoting the $L^2$ inner product,
the weaker version of (A1) is that
\[
{\bf x}(1)\tr{\bf T}{\bf x}(1) +
\langle \m{x}, \m{Qx}\rangle +
\langle \m{u}, \m{Ru}\rangle +
2\langle \m{x}, \m{Su}\rangle \ge
\alpha
\langle \m{u}, \m{u}\rangle,
\]
whenever $(\m{x}, \m{u})$ satisfies
$\dot{\m{x}} = \m{Ax} + \m{Bu}$ with
$\m{x}(-1) = \m{0}$ and $\m{u} = \m{v} - \m{w}$
for some $\m{v}$ and $\m{w} \in L^2$ satisfying
$\m{v}(t)$ and $\m{w}(t) \in \C{U}$ for almost every $t \in [-1, 1]$.
For the Euler integration scheme, we show in \cite[Lem.~11]{DontchevHager93}
that this weaker condition implies an analogous coercivity property
for the discrete problem.
The extension of this result from the Euler scheme to orthogonal collocation
schemes remains an open problem.

Let $\m{D}$ be the $N$ by $N+1$ matrix defined by
\begin{equation}\label{Ddef}
D_{ij} = \dot{L}_j (\tau_i), \;
\mbox{where }
L_j (\tau) := \prod^{N}_{\substack{l=0\\l\neq j}}
\frac{\tau-\tau_l}{\tau_j-\tau_l}, \;
1 \le i \le N \mbox{ and } 0 \le j \le N.
\end{equation}
The matrix $\m{D}$ is a differentiation matrix in the sense that
$(\m{Dp})_i = \dot{p} (\tau_i)$, $1 \le i \le N$,
whenever $p \in \C{P}_N$ is the polynomial that satisfies
$p(\tau_j) = p_j$ for $0 \le j \le N$.
The submatrix $\m{D}_{1:N}$, consisting of the trailing $N$ columns of $\m{D}$,
has the following properties which are utilized in the analysis:
\smallskip
\begin{itemize}
\item [(P1)]
$\m{D}_{1:N}$ is invertible and
$\| \m{D}_{1:N}^{-1}\|_\infty \le 2$.
\item [(P2)]
If $\m{W}$ is the diagonal matrix containing the Gauss
quadrature weights $\omega_i$, $1 \le i \le N$, on the diagonal,
then the rows of the matrix $[\m{W}^{1/2} \m{D}_{1:N}]^{-1}$
have Euclidean norm bounded by $\sqrt{2}$.
\end{itemize}
\smallskip

The invertibility of $\m{D}_{1:N}$ is proved in \cite[Prop. 1]{GargHagerRao10a}.
The bound for the inverse appearing in (P1) is established in
Appendix~1.
(P2) has been checked numerically for $N$ up to 300 in \cite{HagerHouRao16c}.
Some intuition concerning the general validity of (P2) is as follows:
It is observed numerically that the last row of
the matrix $[\m{W}^{1/2} \m{D}_{1:N}]^{-1}$ has the largest Euclidean norm
among all the rows.
Based on the formula for $\m{D}_{1:N}^{-1}$ given in
\cite[Sect. 4.1.2]{GargHagerRao10a},
the $i$-th element in the last row approaches $\omega_i$ as $N$ tends
to infinity.
Hence, the $i$-th element in the last row of $[\m{W}^{1/2} \m{D}_{1:N}]^{-1}$
approaches $\sqrt{\omega_i}$ as $N$ tends to infinity.
Since the quadrature weights sum to 2, the Euclidean norm of the last row
of $[\m{W}^{1/2} \m{D}_{1:N}]^{-1}$ should be close to $\sqrt{2}$.
Despite the strong numerical evidence for (P2),
a proof of (P2) for general $N$ is still missing.

The properties (P1) and (P2) are stated separately since they are used
in different ways in the analysis.
However, (P2) implies (P1) by the Schwarz inequality.
That is, if $\m{r}$ is a row from $\m{D}_{1:N}^{-1}$, then we have
\[
\sum_{i=1}^N |r_i| =
\sum_{i=1}^N \sqrt{\omega_i} \left( |r_i|/\sqrt{\omega_i}\right) \le
\left( \sum_{i=1}^N \omega_i \right)^{1/2}
\left( \sum_{i=1}^N r_i^2/\omega_i \right)^{1/2} \le 2
\]
since the quadrature weights sum to 2 and when (P2) holds,
the Euclidean norm of a row from $[\m{W}^{1/2} \m{D}_{1:N}]^{-1}$ is at most
$\sqrt{2}$.

If $\m{x}^N \in \C{P}_N^n$ is a solution of (\ref{D}) associated
with the discrete controls $\m{u}_i$, $1 \le i \le N$, and
if $\g{\lambda}^N \in \C{P}_N^n$ satisfies
(\ref{dcostate})--(\ref{dcontrolmin}), then we define
\[
\begin{array}{llllllll}
\m{X}^N &= [ &\m{x}^N(-1), & \m{x}^N(\tau_1), & \ldots,
& \m{x}^N(\tau_N), & \m{x}^N(+1) &], \\
\m{X}^* &= [ &\m{x}^*(-1), & \m{x}^*(\tau_1), & \ldots,
& \m{x}^*(\tau_N), & \m{x}^*(+1) &], \\
\m{U}^N &= [ && \m{u}_1, & \ldots, & \m{u}_N\; & &], \\
\m{U}^* &= [ && \m{u}^*(\tau_1), & \ldots, & \m{u}^*(\tau_N)& & ], \\
\g{\Lambda}^N &= [ &\g{\lambda}^N(-1), & \g{\lambda}^N(\tau_1),
& \ldots, & \g{\lambda}^N(\tau_N), & \g{\lambda}^N(+1) &], \\
\g{\Lambda}^* &= [ &\g{\lambda}^*(-1), & \g{\lambda}^*(\tau_1),
& \ldots, & \g{\lambda}^*(\tau_N), & \g{\lambda}^*(+1) &].
\end{array}
\]
The following convergence result relative to the vector $\infty$-norm
(largest absolute element) is established.
Here $\C{H}^p(\Omega;\; \mathbb{R}^n)$ denotes the
Sobolev space of functions with square integrable
derivatives through order $p$ and norm denoted
$\| \cdot \|_{\C{H}^p(\Omega;\; \mathbb{R}^n)}$.
\smallskip

\begin{theorem}\label{maintheorem}
Suppose $(\m{x}^*, \m{u}^*)$ is a local minimizer for the continuous problem
$(\ref{P})$ with $(\m{x}^*, \g{\lambda}^*) \in \C{H}^\eta(\Omega; \mathbb{R}^n)$
for some $\eta \ge 2$.
If both {\rm (A1)--(A2)} and {\rm (P1)--(P2)} hold,
then for $N$ sufficiently large,
the discrete problem $(\ref{D})$ has a local minimizer
$\m{x}^N \in \C{P}_N^n$ and $\m{u}\in \mathbb{R}^{mN}$, and an associated
multiplier $\g{\lambda}^N \in \C{P}_N^n$ satisfying
$(\ref{dcostate})$--$(\ref{dcontrolmin})$; moreover,
there exists a constant $c$ independent of $N$ and $\eta$ such that
\begin{eqnarray}
&\max \left\{ \|{\bf X}^N-{\bf X}^*\|_\infty ,
\|{\bf U}^N-{\bf U}^*\|_\infty,
\|{\g \Lambda}^N-{\g \Lambda}^*\|_\infty \right\}& \nonumber \\
&\leq \left( \displaystyle{\frac{c}{N}} \right)^{p-3/2} \left(
\|\m{x}^*\|_{\C{H}^p(\Omega;\; \mathbb{R}^n)}
+ \|\g{\lambda}^*\|_{\C{H}^p(\Omega;\; \mathbb{R}^n)}
\right), \quad p := \min\{\eta, N+1\}.&
\label{maineq}
\end{eqnarray}
\end{theorem}
\smallskip

This result was established in \cite{HagerHouRao16c} for unconstrained
control problem, but with the exponent 3/2 replaced by 3 and with $\eta \ge 4$.
Hence, the analysis is extended to control constrained problems
and the exponent of $N$ in the convergence estimate is improved by 1.5.
Since typical control constrained problems have regularity at most
$\eta = 2$ when (A1) holds, there is no guarantee of convergence with
the previous estimate.

The paper is organized as follows.
In Section~\ref{abstract} the discrete optimization
problem (\ref{D}) is reformulated as a differential inclusion obtained
from the first-order optimality conditions,
and a general approach to convergence analysis is presented.
We also establish the connection between the Karush-Kuhn-Tucker conditions
and the polynomial conditions (\ref{dcostate})--(\ref{dcontrolmin}).
In Section~\ref{sect_interp} we use results from \cite{BernardiMaday92} to
bound the derivative of the interpolation error in $\C{L}^2$.
Section~\ref{residual} estimates how closely the
solution to the continuous problem satisfies the first-order optimality
conditions for the discrete problem, while
Section~\ref{inverse} establishes the invertibility of the linearized
dynamics for the discrete problem.
Section~\ref{Lip} proves a Lipschitz property for the linearized
optimality conditions, which yields a proof of Theorem~\ref{maintheorem}.
A numerical example given in Section~\ref{numerical} indicates the
potential for further improvements to the convergence rate exponent.
Section~\ref{appendix2} contains a result of Yvon Maday concerning the
error in best $\C{H}^1$ approximation relative to an
$\C{L}^2$ norm with a singular weight function.

{\bf Notation.}
We let $\C{P}_N$ denote the space of polynomials of degree at most $N$,
while $\C{P}_N^0$ is the subspace consisting of polynomials in
$\C{P}_N$ that vanish at $t = -1$ and $t = 1$.
The Gauss collocation points
$\tau_i$, $1 \le i \le N$, are the roots of the
Legendre polynomial $P_N$ of degree $N$.
The associated Gauss quadrature weights
$\omega_i$, $1 \le i \le N$, are given by
\begin{equation}\label{GaussQuad}
\omega_i = \frac{2}{(1-\tau_i^2) P_N'(\tau_i)^2} .
\end{equation}
For any $p \in \C{P}_{2N-1}$, we have \cite[Thm. 3.6.24]{StoerBulirsch02}
\begin{equation}\label{exact}
\int_\Omega p(t) \, dt = \sum_{i=1}^N \omega_i p(\tau_i) .
\end{equation}
Derivatives with respect to $t$ are denoted with either a dot above
the function as in $\dot{\m{x}}$, which is common in the optimal
control literature, or with an accent as in $p'$,
which is common in the numerical analysis literature.
The meaning of the norm $\| \cdot \|_\infty$ is based on context.
If $\m{x} \in \C{C}^0 (\mathbb{R}^n)$, then
$\|\m{x}\|_\infty$ denotes the maximum of $|\m{x}(t)|$ over
$t \in [-1, 1]$, where $| \cdot|$ is the Euclidean norm.
For a vector $\m{v} \in \mathbb{R}^m$,
$\|\m{v}\|_\infty$ is the maximum of $|v_i|$ over $1 \le i \le m$.
If $\m{A} \in \mathbb{R}^{m \times n}$, then $\|\m{A}\|_\infty$
is the largest absolute row sum
(the matrix norm induced by the vector sup-norm).
We often partition a vector $\m{p} \in \mathbb{R}^{nN}$ into subvectors
$\m{p}_i \in \mathbb{R}^n$, $1 \le i \le N$.
Similarly, if $\m{p} \in \mathbb{R}^{mN}$, then $\m{p}_i \in \mathbb{R}^m$.
The dimension of the identity matrix $\m{I}$ is often clear from context;
when necessary, the dimension of $\m{I}$ is specified by a subscript.
For example, $\m{I}_n$ is the $n$ by $n$ identity matrix.
The gradient is denoted $\nabla$, while $\nabla^2$ denotes the Hessian;
subscripts indicate the differentiation variables.
Throughout the paper, $c$ is a generic constant which is independent of
the polynomial degree $N$ and the smoothness $\eta$, and
which may have different values in different equations.
The vector $\m{1}$ has all entries equal to one, while
the vector $\m{0}$ has all entries equal to zero;
again, their dimension should be clear from context.
If $\m{D}$ is the differentiation matrix introduced in (\ref{Ddef}), then
$\m{D}_j$ is the $j$-th column of $\m{D}$ and
$\m{D}_{i:j}$ is the submatrix formed by columns $i$ through $j$.
We let $\otimes$ denote the Kronecker product.
If $\m{U} \in \mathbb{R}^{m \times n}$ and $\m{V} \in \mathbb{R}^{p \times q}$,
then $\m{U} \otimes \m{V}$ is the $mp$ by $nq$ matrix composed of
$p \times q$ blocks; the $(i,j)$ block is $u_{ij} \m{V}$.
We let $\C{L}^2 (\Omega)$ denote the usual space of functions square integrable
on $\Omega$, while $\C{H}^p(\Omega)$ is the Sobolev space consisting
of functions with square integrable derivatives through order $p$.
The norm in $\C{H}^p(\Omega)$ is denoted $\| \cdot \|_{\C{H}^p(\Omega)}$.
The seminorm in $\C{H}^1(\Omega)$ corresponding to the $\C{L}^2(\Omega)$
norm of the derivative is denoted $| \cdot |_{\C{H}^1(\Omega)}$.
The subspace of $\C{H}^1(\Omega)$ corresponding to functions that vanish
at $t = -1$ and $t = 1$ is denoted $\C{H}_0^1(\Omega)$.
We let $\C{H}^p(\Omega ;\; \mathbb{R}^n)$ denote the $n$-fold Cartesian
product $\C{H}^p(\Omega) \times \ldots \times \C{H}^p(\Omega)$.

\section{Abstract Setting}
\label{abstract}
In the introduction, we formulated the discrete optimization problem (\ref{D})
and the necessary conditions (\ref{dcostate})--(\ref{dcontrolmin})
in polynomial spaces.
However, to prove Theorem~\ref{maintheorem}, we reformulate the
first-order optimality conditions in Cartesian space.
Given a feasible point $\m{x}\in \C{P}_N^n$ and $\m{u}\in \mathbb{R}^{mN}$
for the discrete problem (\ref{D}),
define $\m{X}_{j} = \m{x}(\tau_j)$, $0 \le j \le N+1$,
and $\m{U}_{i} = \m{u}_{i}$, $1 \le i \le N$.
As noted earlier, $\m{D}$ is a differentiation matrix in the sense that
\[
\sum_{j=0}^N D_{ij} \m{X}_{j} = \dot{\m{x}} (\tau_i), \quad
1 \le i \le N.
\]
Since $\dot{\m{x}} \in \C{P}_{N-1}^n$, it follows from the
exactness result (\ref{exact}) for Gaussian quadrature that
when $\m{x}$ satisfies the dynamics of (\ref{D}), we have
\[
\m{X}_{N+1} = \m{x}(1) = \m{x}(-1) + \int_{\Omega} \dot{\m{x}}(t)\, dt =
\m{X}_0 + \sum_{j=1}^N\omega_j{\m f}({\m X}_j,{\m U}_j).
\]
Hence, the discrete problem (\ref{D}) can be reformulated as the
nonlinear programming problem
\begin{eqnarray}
&\mbox {minimize} &\quad C(\m{X}_{N+1}) \nonumber\\
&\mbox {subject to} & \quad \sum_{j=0}^N D_{ij} \m{X}_{j} =
\m{f}(\m{X}_{i}, \m{U}_{i}), \quad \m{U}_i \in \C{U}, \quad 1 \le i \le N,
\label{nlp}\\[-.15in]
&&\quad \m{X}_0=\m{x}_0, \quad
{\m X}_{N+1}={\m X}_0+\sum_{j=1}^N\omega_j{\m f}({\m X}_j,{\m U}_j).
\nonumber
\end{eqnarray}
To prove Theorem~\ref{maintheorem}, we analyze the existence and
stability of solutions to the first-order optimality conditions associated
with the nonlinear programming problem.

We introduce multipliers $\g{\mu}_j \in \mathbb{R}^n$,
$0 \le j \le N+1$ corresponding to each of the constraints in the
nonlinear program.
The first-order optimality conditions correspond to
stationary points of the Lagrangian
\begin{eqnarray*}
&
C(\mathbf{X}_{N+1})&
+\sum_{i=1}^N\left\langle\g{\mu}_{i},
{\bf f}({\bf X}_{i},{\bf U}_{i})
-\sum_{j=0}^{N}{D}_{ij}{\bf X}_{j} \right\rangle +
\left\langle \g{\mu}_0, \m{x}_0 - \m{X}_0 \right\rangle \\
&& + \left\langle \g{\mu}_{N+1} , \m{X}_0 - \m{X}_{N+1}
+ \sum_{i = 1}^N \omega_i \m{f}(\m{X}_i, \m{U}_i) \right\rangle.
\end{eqnarray*}
The stationarity conditions for the Lagrangian appear below.
\begin{eqnarray}
\m{X}_{0}\;&\Rightarrow& \g{\mu}_{N+1} = \g{\mu}_0 +
\sum_{i=1}^N{D}_{i0}{\g \mu}_{i},
\label{NC0} \\
\m{X}_{j}\;&\Rightarrow&
\sum_{i=1}^N {D}_{ij}\g{\mu}_{i} = \nabla_x H
({\bf X}_{j},{\bf U}_{j},\g{\mu}_{j} + \omega_j\g{\mu}_{N+1}),
\quad 1 \le j \le N, \label{NC1} \\
\m{X}_{N+1}&\Rightarrow&
\g{\mu}_{N+1} = \nabla C(\m{X}_{N+1}),
\label{NC2} \\[.1in]
\m{U}_{i}\;&\Rightarrow&
-\nabla_u H\left({\bf X}_{i},{\bf U}_{i},{\g \mu}_{i}
+ \omega_i \g{\mu}_{N+1}\right)
\in N_\C{U}(\m{U}_i), \quad 1 \le i \le N.
\label{NC3}
\end{eqnarray}
Since there are no state constraints, the conditions
(\ref{NC0})--(\ref{NC2}) are obtained by setting to zero the
derivative of the Lagrangian with respect to the indicated variables.
The condition (\ref{NC3}) corresponds to stationarity of the Lagrangian
respect to the control.
The relation between multipliers satisfying
(\ref{NC0})--(\ref{NC3}) and $\g{\lambda} \in \C{P}_N^n$
satisfying (\ref{dcostate})--(\ref{dcontrolmin}) is as follows.
\smallskip

\begin{proposition}\label{equiv}
The multipliers $\g{\mu} \in \mathbb{R}^{n(N+2)}$
satisfy $(\ref{NC0})$--$(\ref{NC3})$ if and only if the polynomial
$\g{\lambda} \in \C{P}_{N}^n$ satisfying the $N+1$ interpolation conditions
$\g{\lambda}(1) = \g{\mu}_{N+1}$ and
$\g{\lambda}(\tau_i) = \g{\mu}_{N+1} + \g{\mu}_{i}/\omega_i$,
$1 \le i \le N$,
is a solution of $(\ref{dcostate})$--$(\ref{dcontrolmin})$ and
$\g{\lambda}(-1) = \g{\mu}_0$.
\end{proposition}
\smallskip

\begin{proof}
We start with multipliers $\g{\mu}$
satisfying (\ref{NC0})--(\ref{NC3}) and show that
$\g{\lambda} \in \C{P}_{N}^n$ satisfying the interpolation conditions
$\g{\lambda}(1) = \g{\mu}_{N+1}$ and
$\g{\lambda}(\tau_i) = \g{\mu}_{N+1} + \g{\mu}_{i}/\omega_i$,
$1 \le i \le N$, is a solution of $(\ref{dcostate})$--$(\ref{dcontrolmin})$
with $\g{\lambda}(-1) = \g{\mu}_0$.
The converse follows by reversing all the steps in the derivation.
Define $\g{\Lambda}_i = \g{\mu}_{N+1} + \g{\mu}_{i}/\omega_i$ for
$1 \le i \le N$, $\g{\Lambda}_{N+1} = \g{\mu}_{N+1}$, and
$\g{\Lambda}_0 = \g{\mu}_0$.
Hence, we have $\g{\mu}_i = \omega_i (\g{\Lambda}_i - \g{\Lambda}_{N+1})$
for $1 \le i \le N$.
In (\ref{NC3}) we divide by $\omega_i$ and substitute
$\g{\Lambda}_i = \g{\mu}_{N+1} + \g{\mu}_{i}/\omega_i$.
In (\ref{NC1}) we divide by $\omega_j$, and substitute
$\g{\Lambda}_j = \g{\mu}_{N+1} + \g{\mu}_{j}/\omega_j$ and
\begin{equation}\label{Ddagger}
{D}_{ij}=- \left(\frac{\omega_j}{\omega_i}\right) {D}_{ji}^\dag, \quad
{D}_{i,N+1}^\dag=-\sum_{j=1}^N{D}_{ij}^\dag,
\quad 1\leq i \leq N.
\end{equation}
With these modifications, (\ref{NC1})--(\ref{NC3}) become
\begin{eqnarray}
\sum_{j=1}^{N+1}{D}_{ij}^\dag{\g \Lambda}_j &=&
-\nabla_x H\left({\m X}_i,{\m U}_i, {\g \Lambda}_i\right), \label{Dadjoint} \\
{\g \Lambda}_{N+1}&=&\nabla C(\m{X}_{N+1}), \label{Dterminal} \\
N_\C{U}(\m{U}_i) &\ni&
-\nabla_u H\left({\m X}_i,{\m U}_i, {\g \Lambda}_i\right) ,
\label{Dcontrolmin}
\end{eqnarray}
$1 \le i \le N$.
In \cite[Thm.~1]{GargHagerRao10a} it is shown that if
$\g{\lambda} \in \C{P}_N^n$ is a polynomial that satisfies the conditions
$\g{\lambda}(\tau_i) = \g{\Lambda}_i$ for $1 \le i \le N+1$, then
\begin{equation}\label{dif}
\sum_{j=1}^{N+1}{D}_{ij}^\dag{\g \Lambda}_j = \dot{\g{\lambda}}(\tau_i),
\quad 1 \le i \le N.
\end{equation}
This identity coupled with (\ref{Dadjoint})--(\ref{Dcontrolmin}) imply that
$(\ref{dcostate})$--$(\ref{dcontrolmin})$ hold.

Now let us consider the final term in (\ref{NC0}).
Since the polynomial that is identically equal to $\m{1}$
has derivative $\m{0}$ and since $\m{D}$ is a differentiation matrix,
we have $\m{D1} = \m{0}$,
which implies that ${\bf D}_{0}=-\sum_{j=1}^N{\bf D}_j$,
where $\m{D}_j$ is the $j$-th column of $\m{D}$.
Hence, the final term in \eqref{NC0} can be written
\begin{eqnarray}
\sum_{i=1}^N\g{\mu}_{i} D_{i0} &=&
-\sum_{i=1}^N\sum_{j=1}^N\g{\mu}_{i}D_{ij} =
-\sum_{i=1}^N\sum_{j=1}^N \omega_j \left( \frac{\g{\mu}_{i}}{\omega_i} \right)
\left( \frac{\omega_i D_{ij}}{\omega_j} \right) \nonumber \\
&=& \sum_{i=1}^N\sum_{j=1}^N
\omega_i D_{ij}^\dag(\g{\Lambda}_{j} - \g{\Lambda}_{N+1})
= \sum_{i=1}^N\sum_{j=1}^{N+1} \omega_i D_{ij}^\dag\g{\Lambda}_{j} .
\label{mu0expand}
\end{eqnarray}
Again, if $\g{\lambda} \in \C{P}_N^n$ is the interpolating polynomial
that satisfies $\g{\lambda}(\tau_i) = \g{\Lambda}_i$ for $1 \le i \le N+1$,
then by (\ref{dif}), (\ref{mu0expand}), and the exactness of Gaussian
quadrature for polynomials in $\C{P}_{N-1}^n$, we have
\begin{equation}\label{Di0}
\sum_{i=1}^N\g{\mu}_{i} D_{i0} =
\sum_{i=1}^N \omega_i \dot{\g{\lambda}}(\tau_i) =
\int_\Omega \dot{\g{\lambda}}(\tau) \, d\tau = \g{\lambda}(1) - \g{\lambda}(-1).
\end{equation}
Since $\g{\lambda}(1) = \g{\lambda}_{N+1} = \g{\mu}_{N+1}$, we deduce
from (\ref{NC0}) and (\ref{Di0}) that $\g{\lambda}(-1) = \g{\mu}_0$.
\end{proof}
\smallskip

In the proof of Proposition~\ref{equiv}, $\g{\Lambda}_0 = \g{\mu}_0$ and
$\g{\Lambda}_{N+1} = \g{\mu}_{N+1}$.
We combine (\ref{NC0}), (\ref{Dadjoint}), and (\ref{mu0expand}) to obtain
\begin{equation}\label{mu0}
\g{\Lambda}_{N+1} = \g{\Lambda}_0 - \sum_{i=1}^N \omega_i \nabla_x
H(\m{X}_i, \m{U}_i, \g{\Lambda}_i) .
\end{equation}
Based on Proposition~\ref{equiv}, the optimality conditions
(\ref{NC0})--(\ref{NC3}) are equivalent to
(\ref{dcostate})--(\ref{dcontrolmin}), which are equivalent to
(\ref{Dadjoint})--(\ref{Dcontrolmin}) and (\ref{mu0}).
This latter formulation, which we refer to as the transformed adjoint system
in our earlier work \cite{Hager99c}, is most convenient for the
subsequent analysis.
This leads us to write the first-order optimality conditions for (\ref{D})
as an inclusion $\C{T}(\m{X}, \m{U}, \g{\Lambda}) \in \C{F}(\m{U})$ where
\[
(\C{T}_0, \C{T}_1, \ldots, \C{T}_6) (\m{X}, \m{U}, \g{\Lambda}) \in
\mathbb{R}^n \times \mathbb{R}^{nN} \times \mathbb{R}^{n} \times
\mathbb{R}^{n} \times \mathbb{R}^{nN} \times \mathbb{R}^{n} \times
\mathbb{R}^{mN}.
\]
The 7 components of $\C{T}$ are defined as
\begin{eqnarray*}
\C{T}_{0}(\m{X}, \m{U}, \g{\Lambda}) &=&
\m{X}_0 - \m{x}_0, \\
\C{T}_{1i}(\m{X}, \m{U}, \g{\Lambda}) &=&
\left( \sum_{j=0}^{N}{D}_{ij}{\bf X}_{j} \right)
-{\bf f}({\bf X}_{i},{\bf U}_{i}),
\quad 1\leq i\leq N ,\\
\C{T}_{2}(\m{X}, \m{U}, \g{\Lambda}) &=&
{\m X}_{N+1}-{\m X}_0-\sum_{j=1}^N\omega_j{\m f}({\m X}_j,{\m U}_j), \\
\C{T}_{3}(\m{X}, \m{U}, \g{\Lambda}) &=&
\g{\Lambda}_{N+1} - \g{\Lambda}_{0} + \sum_{i=1}^N
\omega_i \nabla_x H(\m{X}_{i}, \m{U}_{i}, \g{\Lambda}_{i}),\\
\C{T}_{4i}(\m{X}, \m{U}, \g{\Lambda}) &=&
\left( \sum_{j=1}^{N+1}{D}_{ij}^\dagger{\g \Lambda}_{j} \right) +
\nabla_x H({\bf X}_{i},{\bf U}_{i}, {\g \Lambda}_{i}),
\quad 1 \leq i \le N , \\
\C{T}_{5}(\m{X}, \m{U}, \g{\Lambda}) &=&
\g{\Lambda}_{N+1} - \nabla C(\m{X}_{N+1}), \\
\C{T}_{6i}(\m{X}, \m{U}, \g{\Lambda}) &=&
-\nabla_u H({\bf X}_{i}, {\bf U}_{i}, {\g \Lambda}_{i}),
\quad 1\leq i\leq N .
\end{eqnarray*}
The components of $\C{F}$ are given by
\[
\C{F}_0 = \C{F}_1 = \ldots = \C{F}_5 = \m{0}, \quad
\mbox{while }\C{F}_{6i}(\m{U}) = N_\C{U}(\m{U}_i).
\]

The first three components of the inclusion
$\C{T}(\m{X}, \m{U}, \g{\Lambda}) \in \C{F}(\m{U})$ are the constraints
of (\ref{nlp}), the next three components describe the discrete costate
dynamics, and the last component is the discrete version of the
Pontryagin minimum principle.
The proof of Theorem~\ref{maintheorem} is based on an existence and stability
result for local solutions of the inclusion
$\C{T}(\m{X}, \m{U}, \g{\Lambda}) \in \C{F}(\m{U})$.
We will apply 
\cite[Proposition 3.1]{DontchevHagerVeliov00}, which is repeated below
for convenience.
Other results like this are contained in
\cite[Thm.~3.1]{DontchevHager97}, in \cite[Thm.~1]{Hager90},
in \cite[Prop.~5.1]{Hager99c}, and in \cite[Thm.~2.1]{Hager02b}.
\smallskip

\begin{proposition}\label{abstractProp}
Let ${\cal X}$ be a Banach space and
let ${\cal Y}$ be a linear normed space with the norms in both
spaces denoted $\| \cdot \|$.
Let ${\cal F} : {\cal X} \mapsto 2^{\cal Y}$
and let ${\cal T}: {\cal X} \mapsto {\cal Y}$ with $\cal T$
continuously Fr\'{e}chet differentiable
in $B_r(\g{\theta}^*)$ for some $\g{\theta}^* \in {\cal X}$ and $r > 0$.
Suppose that the following conditions hold for some
$\g{\delta} \in {\cal Y}$ and scalars $\epsilon$ and $\gamma > 0$:
\begin{itemize}
\item[{\rm (C1)}]
${\cal T}(\g{\theta}^*) + \g{\delta} \in {\cal F}(\g{\theta}^*)$.
\item[{\rm (C2)}]
$\|\nabla {\cal T}(\g{\theta}) - \nabla \C{T}(\g{\theta}^*)\| \le \epsilon$
for all $\g{\theta} \in B_r(\g{\theta}^*)$.
\item[{\rm (C3)}]
The map $({\cal F} - \nabla \C{T}(\g{\theta}^*))^{-1}$ is single-valued
and Lipschitz continuous with Lipschitz constant $\gamma$.
\end{itemize}
If $\epsilon \gamma < 1$ and
$\|\g{\delta}\| \le (1 - \gamma \epsilon )r/\gamma$,
then there exists a unique $\g{\theta} \in B_r(\g{\theta}^*)$ such that
${\cal T}(\g{\theta}) \in {\cal F}(\g{\theta})$.
Moreover, we have the estimate
\begin{equation} \label{abs}
\|\g{\theta} - \g{\theta}^*\| \leq
\frac{\gamma}{1 - \gamma \epsilon} \|\g{\delta} \| .
\end{equation}
\end{proposition}
\smallskip

We apply Proposition~\ref{abstractProp} with
$\g{\theta}^* = (\m{X}^*, \m{U}^*, \g{\Lambda}^*)$ and
$\g{\theta} = (\m{X}^N, \m{U}^N, \g{\Lambda}^N)$, where
the discrete variables were defined before Theorem~\ref{maintheorem}.
The key steps in the analysis are the estimation of the residual
$\left\|\mathcal{T}\left(\g{\theta}^*\right)\right\|$,
the proof that
$\C{F} - \nabla\mathcal{T}(\g{\theta}^*)$ is invertible,
and the proof that $(\C{F} - \nabla\mathcal{T}(\g{\theta}^*))^{-1}$
is Lipschitz continuous with respect to the norms in $\C{X}$ and $\C{Y}$.
In our context, we use the sup-norm for $\C{X}$.
In particular,
\[
\|\g \theta\|=\|({\bf X},{\bf U}, {\g \Lambda})\|_\infty
=\max \left\{\|\bf X\|_\infty, \|\bf U\|_\infty,\|\g \Lambda\|_\infty \right\}.
\]
For this norm, the left side of (\ref{maineq}) and the left side of (\ref{abs})
are the same.
The norm on $\C{Y}$ enters into the estimation of both the distance from
$\|\mathcal{T}(\g{\theta}^*)\|$ to $\C{F}(\g{\theta}^*)$
($\|\g{\delta}\|$ in (\ref{abs})) and the Lipschitz constant $\gamma$ for
$(\C{F} - \nabla\C{T}(\g{\theta}^*))^{-1}$.
In our context,
we think of an element of $\C{Y}$ as a large vector with components
$\m{y}_l \in \mathbb{R}^n$ or $\mathbb{R}^m$.
There are $N$ components in $\mathbb{R}^m$ associated with
$\C{T}_6$, one component in $\mathbb{R}^n$ associated with
each of $\C{T}_0$, $\C{T}_2$, $\C{T}_3$, and $\C{T}_5$,
and $N$ components in $\mathbb{R}^n$ associated with $\C{T}_1$ and $\C{T}_4$.
Hence, $\C{Y}$ has dimension $mN + 4n + 2nN$ which matches the dimension
of $\C{X}$ since dim$(\m{U})$ = $mN$, dim$(\m{X})$ = $(n+2)N$, and
dim$(\g{\Lambda})$ = $(n+2)N$.
For the norm of $\m{y} \in \C{Y}$, we take
\[
\|\m{y}\|_\C{Y} = |\m{y}_0|
+ |\m{y}_2|
+ |\m{y}_3|
+ |\m{y}_5|
+ \|\m{y}_6\|_\infty
+ \|\m{y}_1\|_{\omega}
+ \|\m{y}_4\|_{\omega} .
\]
Here $\omega$-norm used for $\m{y}_1$
(state dynamics) and $\m{y}_4$ (costate dynamics) is defined by
\[
\|\m{z}\|_{\omega}^2 =
\left( \sum_{i=1}^N \omega_i |\m{z}_{i}|^2 \right)^{1/2},
\quad \m{z} \in \mathbb{R}^{nN}.
\]
Note that the $\omega$-norm has the upper bound
\begin{equation}\label{normbound}
\|\m{z}\|_{\omega} \le \sqrt{2n} \|\m{z}\|_\infty
\end{equation}
since the $\omega_i$ are positive and sum to 2.
\section{Interpolation error in $\C{H}^1$}
\label{sect_interp}

Our error analysis is based on a result concerning the error in
interpolation at the point set $\tau_i$, $0 \le i \le N$,
where $\tau_i$ for $i > 0$ are the $N$ Gauss
quadrature points on $\Omega$, and $\tau_0 = -1$.
In \cite[Thm.~4.8]{BernardiMaday92},
Bernardi and Maday give an overview of the analysis of
error in $\C{H}^1$ for interpolation at Gauss quadrature points.
Here we take into account the additional interpolation point $\tau_0 = -1$,
and provide a complete derivation of the interpolation error estimate.
\smallskip
\begin{lemma}\label{interp}
If $u \in \C{H}^\eta(\Omega)$ for some $\eta \ge 1$,
then there exists a constant $c$,
independent of $N$ and $\eta$, such that
\begin{equation}\label{interperror}
|u - u^I|_{\C{H}^1(\Omega)} \le
(c/N)^{p - 3/2} \|u\|_{\C{H}^p(\Omega )},
\quad p = \min \{ \eta, N+1 \},
\end{equation}
where $u^I \in \C{P}_N$ is the interpolant of $u$
satisfying $u^I(\tau_i) = u(\tau_i)$, $0 \le i \le N$, and $N > 0$.
\end{lemma}
\smallskip

\begin{proof}
Throughout the analysis, $c$ denotes a generic constant whose value is
independent of $N$ and $\eta$, and which may have different values in
different equations.
Let $\ell$ denote the linear function for which $\ell(\pm 1) = u(\pm 1)$.
If the lemma holds for all $u \in \C{H}_0^1(\Omega) \cap \C{H}^\eta(\Omega)$,
then it holds for all $u \in \C{H}^\eta(\Omega)$ since
$|u - u^I|_{\C{H}^1(\Omega)} =$
$|(u -\ell) - (u - \ell)^I|_{\C{H}^1(\Omega)}$ and
$\|u-\ell\|_{\C{H}^p(\Omega )} \le$ $c \|u\|_{\C{H}^p(\Omega )}$.
Hence, without loss of generality, it is assumed
that $u \in \C{H}_0^1(\Omega) \cap \C{H}^\eta(\Omega)$.

Let $\pi_N u$ denote the projection of $u$
into $\C{P}_N^0$ relative to the norm $| \cdot |_{\C{H}^1(\Omega)}$.
Define $E_N = u - \pi_N u$ and
$e_N = E_N^I = (u - \pi_N u)^I = u^I - \pi_N u$.
Since $E_N - e_N = u - u^I$, it follows that
\begin{equation}\label{part0}
| u - u^I|_{\C{H}^1(\Omega)} \le
|E_N|_{\C{H}^1(\Omega)} + |e_N|_{\C{H}^1(\Omega)} .
\end{equation}
In \cite[Prop.~3.1]{Elschner93} it is shown that
\begin{equation}\label{part1}
|E_N|_{\C{H}^1(\Omega)} \le (c/N)^{p-1} \|u\|_{\C{H}^p(\Omega)},
\quad \mbox{where } p = \min\{\eta, N+1\}.
\end{equation}
We establish below the bound
\begin{equation}\label{part2}
|e_N|_{\C{H}^1(\Omega)}  \le c \sqrt{N} |E_N|_{\C{H}^1(\Omega)}.
\end{equation}
Estimate (\ref{interperror}) follows, for
an appropriate choice of $c$, by combining (\ref{part0})--(\ref{part2}).

The proof of (\ref{part2}) proceeds as follows:
Let $\phi_N$ be defined by
\begin{equation}\label{phiN}
\phi_{N}(\tau)=e_N(\tau) -e_N(1)w_{N}(\tau), \quad \mbox{where} \quad
w_N(\tau)=\frac{(1+\tau)P_N'(\tau)}{N(N+1)}.
\end{equation}
Since $P_N$, the Legendre polynomial of degree $N$, satisfies
$P_N'(1) =$ $N(N+1)/2$, it follows that $w_N(1) = 1$ and $\phi_N(1) = 0$.
Moreover, since $w_N(-1) = 0$ and $e_N(-1) = e_N(\tau_0) = 0$,
we conclude that $\phi_N(-1) = 0$ and $\phi_N \in \C{P}_N^0$.
In \cite[Lem.~4.4]{BernardiMaday92} it is shown that
any $\phi_N \in \C{P}_N^0$ satisfies
\[
|\phi_N|_{\C{H}^1(\Omega)} \le cN \left( \int_\Omega
\frac{\phi_N^2(\tau)}{1-\tau^2} \;d\tau \right)^{1/2} .
\]
Hence, by (\ref{phiN}), we have
\begin{equation}\label{uIbound}
|e_N|_{\C{H}^1(\Omega)} \le
cN \left( \int_\Omega
\frac{\phi_N^2(\tau)}{1-\tau^2} \;d\tau \right)^{1/2} +
|w_N|_{\C{H}^1(\Omega)} |e_N(1)| .
\end{equation}

Rodrigues' formula for $P_N$ and integration by parts give
\[
\| P_N'\|_{\C{L}^2(\Omega)} = \sqrt{N(N+1)}.
\]
It follows that
\[
\|w_N\|_{\C{L}^2(\Omega)} \le \frac{2}{\sqrt{N(N+1)}} \le \frac{2}{N}.
\]
Bellman's \cite{Bellman44} inequality
\[
\int_\Omega p'(\tau)^2\, dt \le \frac{(N+1)^4}{2} \int_\Omega p(\tau)^2 \, dt
\quad \mbox{for all } p \in \C{P}_N
\]
implies that
\[
|w_N|_{\C{H}^1(\Omega)} =
\|w_N'\|_{\C{L}^2(\Omega)} \le \frac{\sqrt{2}(N+1)^2}{N} \le cN.
\]
We combine this bound for $|w_N|_{\C{H}^1(\Omega)}$ with
(\ref{uIbound}) to obtain
\begin{equation}\label{uIbound2}
|e_N|_{\C{H}^1(\Omega)} \le cN \left[
\left( \int_\Omega \frac{\phi_N^2(\tau)}{1-\tau^2} \;d\tau \right)^{1/2}
+ |e_N(1)| \right] .
\end{equation}

Since $e_N = E_N^I$ and
$E_N (-1) = 0$, the interpolant can be expressed
\[
e_N(\tau) = E_N^I(\tau) = 
\sum_{i=1}^N E_N(\tau_i) \left( \frac{(\tau+1) P_N(\tau)}
{(\tau_i +1) P_N'(\tau_i)(\tau-\tau_i)}\right) ,
\]
where the expression in parentheses is the Lagrange interpolating polynomial;
it vanishes at $\tau = \tau_j$ for $0 \le j \le N$ and $j \ne i$ since the
numerator vanishes, while it is one at $\tau = \tau_i$ since
\[
\left. \frac{P_N(\tau)/(\tau - \tau_i)}
{P_N'(\tau)} \right|_{\tau = \tau_i} =
\frac{ \prod_{j \ne i} (\tau_i - \tau_j)}
{\prod_{j \ne i} (\tau_i - \tau_j)} = 1.
\]
At $\tau = 1$, it follows from the Schwarz inequality that
\[
|e_N(1)| \le \sum_{i=1}^N \frac{2|E_N(\tau_i)|}
{(1-\tau_i^2) |P_N'(\tau_i)|} \le
2 \sqrt{N} \left( \sum_{i=1}^N \frac{E_N^2(\tau_i)}
{(1-\tau_i^2)^2 P_N'(\tau_i)^2} \right)^{1/2}.
\]
Replace $2/[(1-\tau_i^2)P_N'(\tau_i)^2]$ by $\omega_i$ using
(\ref{GaussQuad}) to obtain
\begin{equation}\label{e1}
|e_N(1)|\le \sqrt{2N} \left( \sum_{i=1}^N \frac{\omega_i E_N^2(\tau_i)}
{1-\tau_i^2} \right)^{1/2} .
\end{equation}
Since  $E_N \in \C{H}_0^1(\Omega)$, it follows from
\cite[Lem.~4.3]{BernardiMaday92} that
\begin{equation}\label{quad}
\left(\sum_{i=1}^N
\frac{\omega_i E_N^2(\tau_i)}{1-\tau_i^2} \right)^{1/2} \le
c\left[ \left( \int_\Omega \frac{E_N^2(\tau)}{1-\tau^2} \; d\tau\right)^{1/2}
+ N^{-1} |E_N|_{\C{H}^1(\Omega)} \right] .
\end{equation}
By Proposition~\ref{singular} in Appendix~2,
\begin{equation}\label{A}
N \left[ \left( \int_\Omega \frac{E_N^2(\tau)}{1-\tau^2} \;
d\tau\right)^{1/2}
+ N^{-1} |E_N|_{\C{H}^1(\Omega)} \right] \le 2|E_N|_{\C{H}^1(\Omega)}.
\end{equation}
Together, (\ref{quad}) and (\ref{A}) give
\begin{equation}\label{quad2}
\left(\sum_{i=1}^N
\frac{\omega_i E_N^2(\tau_i)}{1-\tau_i^2} \right)^{1/2} \le
(c/N) |E_N|_{\C{H}^1(\Omega)}.
\end{equation}
Combine (\ref{e1}) and (\ref{quad2}) to obtain
\begin{equation}\label{e2}
|e_N(1)|\le \left( c/\sqrt{N} \right) |E_N|_{\C{H}^1(\Omega)}.
\end{equation}

Since $\phi_N \in \C{P}_N^0$,
we deduce that $\phi_N^2(\tau)/(1-\tau^2) \in \C{P}_{2N-2}$.
Consequently, $N$-point Gaussian quadrature is exact, and we have
\begin{eqnarray*}
\left( \int_\Omega \frac{{\phi}_N^2(\tau)}{1-\tau^2} \;d\tau
\right)^{1/2} &=&
\left(\sum_{i=1}^N
\frac{\omega_i {\phi}_N^2(\tau_i)}{1-\tau_i^2} \right)^{1/2} \\
&\le&
\left(\sum_{i=1}^N
\frac{\omega_i e_N^2(\tau_i)}{1-\tau_i^2} \right)^{1/2} +
|e_N(1)| \left(\sum_{i=1}^N
\frac{\omega_i w_N^2(\tau_i)}{1-\tau_i^2} \right)^{1/2} \\
&=&
\left(\sum_{i=1}^N
\frac{\omega_i E_N^2(\tau_i)}{1-\tau_i^2} \right)^{1/2} +
|e_N(1)| \left(\sum_{i=1}^N
\frac{\omega_i w_N^2(\tau_i)}{1-\tau_i^2} \right)^{1/2} .
\end{eqnarray*}
The last equality holds since $e_N = E_N$ at the collocation points
$\tau_i$, $1 \le i \le N$.
In \cite[(4.15)]{BernardiMaday92}, it is proved that
\begin{equation}\label{B}
\sum_{i=1}^N
\frac{\omega_i w_N^2(\tau_i)}{1-\tau_i^2} \le c.
\end{equation}
Combine (\ref{quad2}), (\ref{e2}), and (\ref{B}) to obtain
\begin{equation}\label{phi2bound}
\left( \int_\Omega \frac{{\phi}_N^2(\tau)}{(1-\tau^2)} \;d\tau
\right)^{1/2} \le (c/\sqrt{N}) |E_N|_{\C{H}^1(\Omega)}.
\end{equation}
Finally, (\ref{uIbound2}), (\ref{e2}), and (\ref{phi2bound}) yield
(\ref{part2}), which completes the proof.
\end{proof}
\section{Analysis of the residual}
\label{residual}
In this section, we establish a bound for the distance from
$\C{T}(\m{X}^*, \m{U}^*, \g{\Lambda}^*)$ to $\C{F}(\m{U}^*)$.
This bound ultimately enters into the right-hand side of the error estimate
(\ref{maineq}).
\begin{lemma}
\label{residual_lemma}
If $\m{x}^*$ and $\g{\lambda}^* \in \C{H}^\eta(\Omega; \;\mathbb{R}^n)$ for
some $\eta \ge 2$, then there exists a constant $c$,
independent of $N$ and $\eta$, such that
\begin{equation}\label{resbound}
{\rm dist}[\C{T}(\m{X}^*,\m{U}^*,\g{\Lambda}^*), \; \C{F}(\m{U}^*)]
\le \left( \frac{c}{N} \right)^{p-3/2}
\left( \|\m{x}^*\|_{\C{H}^p(\Omega)} +
\|\g{\lambda}^*\|_{\C{H}^p(\Omega)} \right) ,
\end{equation}
where $p = \min \{ \eta, N+1 \}$.
The left-hand side of $(\ref{resbound})$ denotes the distance from
$\C{T}(\m{X}^*,\m{U}^*,\g{\Lambda}^*)$ to $\C{F}(\m{U}^*)$ relative to
$\| \cdot \|_\C{Y}$.
\end{lemma}
\smallskip
\begin{proof}
Since $\C{T}(\m{X}^*,\m{U}^*,\g{\Lambda}^*)$ appears throughout the
analysis, it is abbreviated $\C{T}^*$.
The feasibility of $\m{x}^*$ in (\ref{P}) implies that
$\m{X}_0^* = \m{x}_0$, or $\C{T}_0^* = \m{0}$.
By the costate equation (\ref{costate}),
$\g{\Lambda}_{N+1}^* = \g{\lambda}^*(1) =$
$\nabla C(\m{x}^*(1)) = \nabla C(\m{X}^*_{N+1})$, which implies that
$\C{T}_5^* = \m{0}$.
By the Pontryagin minimum principle (\ref{controlmin}),
\[
-\nabla_u H(\m{X}_i^*, \m{U}_i^*, \g{\Lambda}_i^*) =
-\nabla_u H(\m{x}^*(\tau_i), \m{u}^*(\tau_i), \g{\lambda}^*(\tau_i)) \in
\C{F}(\m{u}^*(\tau_i)) = \C{F}(\m{U}^*_i),
\]
$1 \le i \le N$.
Thus $\C{T}_0^* = \C{T}_5^* = \m{0}$,
$\C{T}_{6}^* \in \C{F}_{6}(\m{U}^*)$.

Now let us consider $\mathcal{T}_1$.
Since ${\bf D}$ is a differentiation matrix associated with the
collocation points, we have
\begin{equation}\label{t1}
\sum_{j=0}^{N}D_{ij}
{\bf X}_{kj}^*=\dot{\bf x}^I(\tau_i), \quad 1\leq i \leq N,
\end{equation}
where ${\bf x}^I \in \mathcal{P}_N^n$ is the interpolating polynomial
that passes through ${\bf x}^*(\tau_j)$ for $0 \leq j \leq N$, and
$\dot{\m{x}}^I$ is the derivative of $\m{x}^I$.
Since ${\bf x}^*$ satisfies the dynamics of (\ref{P}),
\begin{equation}\label{t1dotx}
{\bf f}({\bf X}_{i}^*, {\bf U}_{i}^*)=
{\bf f}({\bf x}^* (\tau_i), {\bf u}^* (\tau_i))=
\dot{\bf x}^*(\tau_i).
\end{equation}
Combine \eqref{t1} and \eqref{t1dotx} to obtain
\begin{equation}\label{t1dif}
\mathcal{T}_{1i}^* = \dot{\bf x}^I(\tau_i)-\dot{\bf x}^*(\tau_i), \quad
1 \le i \le N.
\end{equation}
Let $(\dot{\m{x}}^*)^J \in \C{P}_{N-1}^n$
denote the interpolant that passes through
$\dot{\bf x}^*(\tau_i)$ for $1 \le i \le N$.
Since both $\dot{\bf x}^I$ and $(\dot{\m{x}}^*)^J$ are polynomials of
degree $N-1$ and Gaussian quadrature is exact for polynomials of degree
$2N -1$, it follows that
\[
\|\mathcal{T}_{1}^*\|_{\omega} =
\|\dot{\m{x}}^I - (\dot{\m{x}}^*)^J\|_{\C{L}^2(\Omega)} \le
\|\dot{\m{x}}^I - \dot{\m{x}}^*\|_{\C{L}^2(\Omega)} +
\|\dot{\m{x}}^* - (\dot{\m{x}}^*)^J\|_{\C{L}^2(\Omega)}.
\]
By Lemma~\ref{interp},
$\|\dot{\m{x}}^I - \dot{\m{x}}^*\|_{\C{L}^2(\Omega)} \le$
$(c/N)^{p-3/2}\|\m{x}^*\|_{\C{H}^p(\Omega)}$.
By \cite[Cor.~3.2]{BernardiMaday92} and
\cite[Prop.~3.1]{Elschner93}, it follows that
$\|\dot{\m{x}}^* - (\dot{\m{x}}^*)^J\|_{\C{L}^2(\Omega)} \le$
$(c/N)^{p-1}\|\m{x}^*\|_{\C{H}^p(\Omega)}$.
Hence, we have
\begin{equation}\label{h49}
\|\C{T}_1^*\|_{\omega} =
\|\dot{\m{x}}^I - (\dot{\m{x}}^*)^J\|_{\C{L}^2(\Omega)} \le
(c/N)^{p-3/2}\|\m{x}^*\|_{\C{H}^p(\Omega)}.
\end{equation}

The analysis of $\C{T}_4$ is identical to that of $\C{T}_1$, the
only adjustment is that $\g{\lambda}^I$ is the interpolating polynomial
that passes through $\g{\lambda}^*(\tau_j)$ for $1 \le j \le N+1$.
Next, let us consider
\begin{equation}\label{h1}
\C{T}_2^* = \m{x}^* (1) - \m{x}^*(-1)
- \sum_{j=1}^N \omega_j \m{f}(\m{x}^*(\tau_j), \m{u}^*(\tau_j)) .
\end{equation}
By the fundamental theorem of calculus and the
exactness of Gaussian quadrature, we have
\begin{equation} \label{h3}
\m{0} =
{\bf x}^I(1)-{\bf x}^I(-1)-\int_{\Omega}\dot{\bf x}^I(t)\;dt =
{\bf x}^I(1)-{\bf x}^I(-1)-\sum_{j=1}^N\omega_j\dot{\bf x}^I(\tau_j) .
\end{equation}
Subtract (\ref{h3}) from (\ref{h1}) and substitute
$\dot{\bf x}^*(\tau_j) =$
$\m{f}(\m{x}^*(\tau_j), \m{u}^*(\tau_j))$ to obtain
\begin{equation}\label{h4}
\C{T}_2^* =
({\bf x}^*-{\bf x}^I)(1)+\sum_{j=1}^N\omega_j\left(\dot{\bf x}^I(\tau_j)
-\dot{\bf x}^*(\tau_j)\right) .
\end{equation}
Since the $\omega_i$ are positive and sum to 2, it follows from the
Schwarz inequality and (\ref{h49}) that
\begin{eqnarray}
\left| \sum_{j=1}^N\omega_j\left(\dot{\bf x}^I(\tau_j)
-\dot{\bf x}^*(\tau_j)\right) \right| &\le&
\left( \sum_{j=1}^N \omega_i \right)^{1/2}
\left( \sum_{j=1}^N \omega_j\left|\dot{\bf x}^I(\tau_j)
-\dot{\bf x}^*(\tau_j)\right|^2 \right)^{1/2} \nonumber \\
&=& \sqrt{2} \|\dot{\bf x}^I - (\dot{\m{x}}^*)^J\|_{\C{L}^2(\Omega)}
\le (c/N)^{p-3/2}\|\m{x}^*\|_{\C{H}^p(\Omega)}. \quad \quad
\label{h50}
\end{eqnarray}
Also, writing $({\bf x}^*-{\bf x}^I)(1)$ as the integral of the
derivative from $-1$ to 1 and applying the Schwarz inequality yields
\begin{equation}\label{h51}
|\mathbf{x}^{*}(1)-\mathbf{x}^{I}(1)|\le
\sqrt{2}\|\dot{\mathbf{x}}^*-\dot{{\mathbf{x}}}^I\|_{\C{L}^2(\Omega)}
\le (c/N)^{p-3/2}\|\m{x}^*\|_{\C{H}^p(\Omega)},
\end{equation}
where the last equality is by Lemma~\ref{interp}.
Combine (\ref{h4}), (\ref{h50}), and (\ref{h51}) to obtain
$|\C{T}_2^*| \le (c/N)^{p-3/2}\|\m{x}^*\|_{\C{H}^p(\Omega)}$.
The analysis of $\C{T}_3$ is the same as that of $\C{T}_2$.
This completes the proof.
\end{proof}
\smallskip
\section{Invertibility of linearized dynamics}
\label{inverse}
In this section, we introduce the linearized inclusion and established
the invertibility of the linearized dynamics for both the state and costate.
Given $\m{Y} \in \C{Y}$, the linearized problem is to find
$(\m{X}, \m{U}, \g{\Lambda})$ such that
\begin{equation}\label{linearproblem}
\nabla \C{T}(\m{X}^*, \m{U}^*, \g{\Lambda}^*)[\m{X}, \m{U}, \g{\Lambda}] +
\m{Y} \in \C{F}(\m{U}).
\end{equation}
Here $\nabla \C{T}(\m{X}^*, \m{U}^*, \g{\Lambda}^*)[\m{X}, \m{U}, \g{\Lambda}]$
denotes the derivative of $\C{T}$ evaluated at
$(\m{X}^*, \m{U}^*, \g{\Lambda}^*)$ operating on $[\m{X}, \m{U}, \g{\Lambda}]$.
Since $\nabla \C{T}(\m{X}^*, \m{U}^*, \g{\Lambda}^*)$ appears frequently
in the analysis, it is abbreviated $\nabla \C{T}^*$.
This derivative involves the matrices:
\[
{\bf A}_i= 
{\bf A}(\tau_i), \quad
{\bf B}_i= 
{\bf B}(\tau_i), \quad
{\bf Q}_i=
{\bf Q}(\tau_i), \quad
{\bf S}_i=
{\bf S}(\tau_i), \quad
{\bf R}_i=
{\bf R}(\tau_i).
\]
With this notation,
the 7 components of $\nabla \C{T}^*[\m{X}, \m{U}, \g{\Lambda}]$ are as follows:
\begin{eqnarray*}
\nabla \C{T}_{0}^*[\m{X}, \m{U}, \g{\Lambda}] &=&
\m{X}_0 ,\\
\nabla \C{T}_{1i}^*[\m{X}, \m{U}, \g{\Lambda}] &=&
\left( \sum_{j=1}^N{D}_{ij}{\bf X}_j \right)
- \m{A}_i \m{X}_i - \m{B}_i \m{U}_i,
\quad 1\leq i\leq N ,\\
\nabla \C{T}_2^*[\m{X}, \m{U}, \g{\Lambda}] &=&
{\bf X}_{N+1}- \m{X}_0 - \sum_{j=1}^N\omega_j
(\m{A}_j \m{X}_j + \m{B}_j \m{U}_j),
\\
\nabla \C{T}_3^*[\m{X}, \m{U}, \g{\Lambda}] &=&
{\g{\Lambda}}_{N+1}- {\g{\Lambda}}_0 + \sum_{j=1}^N\omega_j
(\m{A}_j\tr \g{\Lambda}_j  + \m{Q}_j \m{X}_j + \m{S}_j \m{U}_j),
\\
\nabla \C{T}_{4i}^*[\m{X}, \m{U}, \g{\Lambda}] &=&
\left( \sum_{j=1}^{N+1}{D}_{ij}^\dag{\g \Lambda}_j \right) +
\m{A}_i\tr \g{\Lambda}_i  + \m{Q}_i \m{X}_i + \m{S}_i \m{U}_i,
\quad 1 \leq i \leq N , \\
\nabla \C{T}_5^*[\m{X}, \m{U}, \g{\Lambda}] &=&
{\g \Lambda}_{N+1} -\m{T} \m{X}_{N+1}, \\[.05in]
\nabla \C{T}_{6i}^*[\m{X}, \m{U}, \g{\Lambda}] &=&
-(\m{S}_i\tr\m{X}_i + \m{R}_i \m{U}_i + \m{B}_i\tr \g{\Lambda}_i),
\quad 1\leq i\leq N .
\end{eqnarray*}

Let us first study the invertibility of the linearized dynamics.
This amounts to solving for the state in (\ref{linearproblem})
for given values of the control.
\smallskip

\begin{lemma}
\label{feasible}
If {\rm (P1)}, {\rm (P2)}, and {\rm (A2)} hold,
then for each $\m{q}_0$ and $\m{q}_1 \in \mathbb{R}^n$ and
$\m{p} \in \mathbb{R}^{nN}$ with $\m{p}_i \in \mathbb{R}^n$,
$1 \le i \le N$, the linear system
\begin{eqnarray}
\left( \sum_{j=0}^{N}{D}_{ij}{\bf X}_j \right) - \m{A}_i \m{X}_i  &=& \m{p}_{i}
\quad 1\leq i\leq N , \label{h99} \\
{\bf X}_{N+1} - \m{X}_0 -\sum_{j=1}^N\omega_j
\m{A}_j \m{X}_j &=&{\m q}_1,
\quad \m{X}_0 = \m{q}_0,
\label{h100}
\end{eqnarray}
has a unique solution $\m{X} \in \mathbb{R}^{n(N+2)}$.
Moreover, there exists a constant $c$, independent of $N$, such that
\begin{equation}\label{xbound}
\|\m{X}\|_\infty \le c (|\m{q}_0| + |\m{q}_1| + \|\m{p}\|_\omega)
\end{equation}
\end{lemma}
\begin{proof}
Let $\bar{\m{X}}$ be the vector obtained by vertically stacking
$\m{X}_1$ through $\m{X}_N$,
let ${\m{A}}$ be the block diagonal matrix
with $i$-th diagonal block $\m{A}_i$, $1 \le i \le N$,
and define
$\bar{\m{D}} = {\bf D}_{1:N}\otimes {\bf I}_n$ where
$\otimes$ is the Kronecker product.
With this notation, the linear system (\ref{h99}) can be expressed
\begin{equation}\label{statedynamics}
(\bar{\m{D}} - {\m{A}}) \bar{\m{X}} = \m{p} - (\m{D}_0\otimes\m{I}_n) \m{q}_0.
\end{equation}
Here $\m{D}_0$ is the first column of $\m{D}$ and the $\m{X}_0 = \m{q}_0$
component of $\m{X}$ has been moved to the right side of the equation.
By (P1) ${\bf D}_{1:N}$ is invertible, which implies that
$\bar{\m{D}}$ is invertible with
$\bar{\m{D}}^{-1} = {\bf D}_{1:N}^{-1}\otimes {\bf I}_n$.
Moreover,
$\|\bar{{\bf D}}^{-1}\|_\infty =$ $\|{\bf D}_{1:N}^{-1}\|_\infty \le 2$ by (P1).
By (A2) $\|{\m{A}}\|_\infty \le \beta$ and
$\|\bar{\m{D}}^{-1}{\m{A}}\|_\infty \le$
$\|\bar{\m{D}}^{-1}\|_\infty \|{\m{A}}\|_\infty \le 2\beta < 1$ since
$\beta < 1/2$.
By \cite[p. 351]{HJ12},
${\bf I}-\bar{\bf D}^{-1}{\bf A}$ is invertible and
\begin{equation}\label{Xbound}
\|({\bf I}-\bar{\bf D}^{-1}{\bf A})^{-1}\|_\infty \leq 1/(1-2\beta).
\end{equation}
%
Multiply (\ref{statedynamics}) first by
$\bar{\bf D}^{-1}$ and then by
$(\m{I} - \bar{\m{D}}^{-1}{\m{A}})^{-1}$ to obtain
\[
\bar{\m{X}} =
({\bf I}-\bar{\bf D}^{-1}{\bf A})^{-1}
\left( \bar{\m{D}}^{-1}\m{p} +
\bar{\m{D}}^{-1} (\m{D}_0\otimes \m{I}_n) \m{q}_0 \right) .
\]
We take the norm of $\bar{\m{X}}$ and utilize (\ref{Xbound})
to find that
\begin{equation}\label{h98}
\|\bar{\m{X}}\|_\infty \le
\left( \frac{1}{1-2\beta} \right)
\left[ \|\bar{\m{D}}^{-1}\m{p}\|_\infty +
\|\bar{\m{D}}^{-1} (\m{D}_0\otimes \m{I}_n) \m{q}_0 \|_\infty \right] .
\end{equation}

Since the polynomial that is identically equal to $\m{1}$
has derivative $\m{0}$ and since $\m{D}$ is a differentiation matrix,
we have $\m{D1} = \m{0}$,
which implies that $\m{D}_{1:N}\m{1} = -{\bf D}_{0}$.
Hence, $\m{D}_{1:N}^{-1} \m{D}_0 = -\m{1}$.
It follows that
\[
(\bar{\m{D}})^{-1}
[\m{D}_{0} \otimes \m{I}_n] =
[({\bf D}_{1:N})^{-1}\otimes {\bf I}_n]
[\m{D}_{0} \otimes \m{I}_n] =
-\m{1}\otimes \m{I}_n .
\]
We make this substitution in (\ref{h98}) and use the bound
for the sup-norm in terms of the Euclidean norm to obtain
\[
\|\bar{\m{X}}\|_\infty \le
\left( \frac{1}{1-2\beta} \right)
\left( \|\bar{\m{D}}^{-1}\m{p}\|_\infty + |\m{q}_0| \right) .
\]

Observe that
\[
\bar{\m{D}}^{-1}\m{p} = 
\left( {\bf D}_{1:N}^{-1}\otimes {\bf I}_n \right) \m{p} =
\left( {\bf D}_{1:N}^{-1} \m{W}^{-1/2} \otimes {\bf I}_n \right)
\left[ \left( \m{W}^{1/2} \otimes \m{I}_n \right) \m{p} \right] ,
\]
where $\m{W}$ is the diagonal matrix with the quadrature weights on
the diagonal.
Based on this identity, an element of $\bar{\m{D}}^{-1}\m{p}$
is the dot product between
\medskip

\begin{center}
a row of $\left( {\bf D}_{1:N}^{-1} \m{W}^{-1/2} \otimes {\bf I}_n \right)$
and the column vector $\left( \m{W}^{1/2} \otimes \m{I}_n \right) \m{p}$.
\end{center}
\medskip
By the Schwarz inequality, this dot product is bounded by the product between
largest Euclidean length of the rows of the matrix and the Euclidean length
of the vector.
By (P2), the Euclidean lengths of the rows of
$[\m{W}^{1/2}{\bf D}_{1:N}]^{-1}$ are bounded by $\sqrt{2}$,
and by the definition of the $\omega$-norm, we have
$|\left( \m{W}^{1/2} \otimes \m{I}_n \right) \m{p}| = \|\m{p}\|_\omega$.
Hence, we have
\begin{equation}\label{h101}
\|\bar{\m{D}}^{-1}\m{p}\|_\infty \le
\sqrt{2} \|\m{p}\|_\omega \quad \mbox{and} \quad
\|\bar{\m{X}}\|_\infty \le \left( \frac{1}{1-2\beta} \right)
\left( \sqrt{2} \|\m{p}\|_\omega + |\m{q}_0| \right) .
\end{equation}

By the first equation in (\ref{h100}),
\[
\|\m{X}_{N+1}\|_\infty \le
\|\m{q}_0\|_\infty +\|\m{q}_1\|_\infty +
\sum_{j=1}^N \omega_j \|\m{A}_j\|_\infty \|\m{X}_j\|_\infty .
\]
Since the $\omega_j$ sum to 2, $\|\m{A}_j\| \le \beta < 1/2$, and the sup-norm
is bounded by the Euclidean norm, it follows that
\begin{equation}\label{h102}
\|\m{X}_{N+1}\|_\infty \le
|\m{q}_0| +|\m{q}_1| + \|\bar{\m{X}}\|_\infty.
\end{equation}
Combine (\ref{h101}) and (\ref{h102}) to obtain (\ref{xbound}).
\end{proof}
\smallskip

Next, let us consider the linearized costate dynamics.
\smallskip

\begin{lemma}
\label{cofeasible}
If {\rm (P1)}, {\rm (P2)}, and {\rm (A2)} hold,
then for each $\m{q}_0$ and $\m{q}_1 \in \mathbb{R}^n$ and
$\m{p} \in \mathbb{R}^{nN}$ with $\m{p}_i \in \mathbb{R}^n$,
$1 \le i \le N$, the linear system
\begin{eqnarray}
\left( \sum_{j=1}^{N+1}{D}_{ij}^\dagger
\g{\Lambda}_j \right) + \m{A}_i\tr \g{\Lambda}_i  &=& \m{p}_{i}
\quad 1\leq i\leq N , \label{h199} \\
\g{\Lambda}_{N+1} - \g{\Lambda}_0 +\sum_{j=1}^N\omega_j
\m{A}_j\tr \g{\Lambda}_j &=&{\m q}_0,
\quad \g{\Lambda}_{N+1} = \m{q}_1,
\label{h200}
\end{eqnarray}
has a unique solution $\g{\Lambda} \in \mathbb{R}^{n(N+2)}$.
Moreover, there exists a constant $c$, independent of $N$, such that
\begin{equation}\label{lbound}
\|\g{\Lambda}\|_\infty \le c (|\m{q}_0| + |\m{q}_1| + \|\m{p}\|_\omega)
\end{equation}
\end{lemma}
\smallskip

\begin{proof}
As noted in (\ref{dif}), $\m{D}^\dagger$ is a differentiation matrix,
analogous to $\m{D}$, except that $\m{D}^\dagger$ operates on function
values at $\tau_1, \ldots , \tau_{N+1}$, while $\m{D}$ operates on
function values at $\tau_0, \ldots, \tau_N$.
The proof is identical to that of Lemma~\ref{feasible} except that
$\g{\Lambda}_{N+1}$ plays the role of $\m{X}_0$, while
$\g{\Lambda}_0$ plays the role of $\m{X}_{N+1}$.
\end{proof}
\smallskip

\section{Invertibility of $\C{F} - \nabla \C{T}^*$ and Lipschitz continuity
of the inverse}
\label{Lip}
The invertibility of $\C{F} - \nabla \C{T}^*$ is now established.
\smallskip

\begin{proposition}\label{invertible}
If {\rm (A1)}--{\rm (A2)} and {\rm (P1)}--{\rm (P2)} hold, then
for each $\m{Y} \in \C{Y}$, there is a unique solution
$(\m{X}, \m{U}, \g{\Lambda})$ to $(\ref{linearproblem})$.
\end{proposition}
\smallskip
\begin{proof}
As in our earlier work
\cite{HagerDontchevPooreYang95,DontchevHager93,DontchevHager98a,
DontchevHagerVeliov00,Hager90,HagerHouRao16a,HagerHouRao15c,HagerHouRao16c},
we formulate a strongly convex quadratic programming
problem whose first-order optimality conditions reduce to
(\ref{linearproblem}).
Let us consider the problem
\begin{equation}\label{QP}
\begin{array}{cl}
\mbox {minimize} &\frac{1}{2} \mathcal{Q}({\bf X},{\bf U})
+ \C{L}(\m{X}, \m{U}, \m{Y}) \\[.08in]
\mbox {subject to} &\sum_{j=1}^N{D}_{ij}{\bf X}_j
={\bf A}_i{\bf X}_i+ {\bf B}_i{\bf U}_i-{\bf y}_{1i}, \quad
\m{U}_i \in \C{U}, \quad 1\leq i \leq N, \\
&\m{X}_0 = -\m{y}_0,
\quad {\bf X}_{N+1}= \m{X}_0 - \m{y}_2 + \sum_{j=1}^N\omega_j
\left({\bf A}_j{\bf X}_j+{\bf B}_j{\bf U}_j
\right) .
\end{array}
\end{equation}
Here the quadratic and linear terms in the objective are
\begin{eqnarray*}
\mathcal{Q}({\bf X},{\bf U})&=&
{\bf X}_{N+1}\tr{\bf T}{\bf X}_{N+1} +
\sum_{i=1}^N\omega_i
\left({\bf X}_i\tr{\bf Q}_i{\bf X}_i
+2{\bf X}_i\tr{\bf S}_i{\bf U}_i
+{\bf U}_i\tr{\bf R}_i{\bf U}_i\right), \\
\C{L}(\m{X}, \m{U}, \m{Y}) &=&
\m{X}_0\tr \left(\m{y}_3 - \sum_{i=1}^N \omega_i \m{y}_{4i} \right)
- \m{y}_5\tr \m{X}_{N+1} + \sum_{i=1}^N\omega_i
\left( \m{y}_{4i}\tr \m{X}_i - \m{y}_{6i}\tr \m{U}_i \right) .
\end{eqnarray*}
In (\ref{QP}), the minimization is over $\m{X}$ and $\m{U}$, while
$\m{Y}$ is a fixed parameter.
By Lemma~\ref{feasible}, the quadratic program (\ref{QP}) is feasible
for any choice of $\m{y}_0$, $\m{y}_1$, and $\m{y}_2$.
Since $\m{X}_0 = -\m{y}_0$, $\m{X}_0$ can be eliminated from the
quadratic program (\ref{QP}).
By (A1), the quadratic program is strongly convex with respect to
$\m{X}_1$, $\ldots$, $\m{X}_{N+1}$, and $\m{U}_1$, $\ldots$, $\m{U}_N$.
Hence, there exists a unique state and control solving (\ref{QP}).
Next, we will show that the first-order optimality conditions for (\ref{QP})
reduce to (\ref{linearproblem}).
These conditions hold since $\C{U}$ has nonempty interior and the
state dynamics have full row rank by Lemma~\ref{feasible}.
Due to the convexity of the objective and constraints,
the first-order optimality conditions are both necessary and sufficient
for optimality.
Uniqueness of $\m{X}$ and $\m{U}$ is due to (A1) and the strong
convexity of (\ref{QP}).
Uniqueness of $\g{\Lambda}$ is by Lemma~\ref{cofeasible}.

Now let us show that (\ref{linearproblem}) corresponds to the
optimality conditions for (\ref{QP}).
Components 0, 1, and 2 of (\ref{linearproblem})
are simply the constraints of (\ref{QP}).
The remaining optimality conditions are associated with the
Lagrangian $L$ given by
\begin{eqnarray*}
&L(\g{\mu}, \m{X}, \m{U}) =
{\textstyle\frac{1}{2}}\C{Q}(\m{X}, \m{U}) + \C{L}(\m{X}, \m{U}, \m{Y}) 
+ \sum_{i=1}^N\left\langle\g{\mu}_{i},
{\bf A}_{i}{\bf X}_{i}+ {\bf B}_{i}{\bf U}_{i} - \m{y}_{1i}
-\sum_{j=0}^{N}{D}_{ij}{\bf X}_{j} \right\rangle & \\
&
- \langle \g{\mu}_0 , \m{X}_0 + \m{y}_1 \rangle
+ \left\langle \g{\mu}_{N+1}, 
\m{X}_0 - \m{X}_{N+1} - \m{y}_2 +
\sum_{j=1}^N\omega_j \left({\bf A}_j{\bf X}_j+{\bf B}_j{\bf U}_j \right)
\right\rangle . &
\end{eqnarray*}
The negative derivative of the Lagrangian with respect to $\m{U}_{i}$ is
\[
\omega_i \left( \m{y}_{6i} - \m{S}_i\tr \m{X}_i - \m{R}_i\m{U}_i -
\m{B}_i\tr \g{\mu}_{N+1} \right) - \m{B}_i\tr \g{\mu}_i .
\]
Substitute $\g{\mu}_{N+1} = \g{\Lambda}_{N+1}$ and
$\g{\mu}_i = \omega_i(\g{\Lambda}_i - \g{\Lambda}_{N+1})$,
$1 \le i \le N$.
The requirement that the resulting vector lies in $N_{\C{U}}(\m{U}_i)$
is the 6-th component of (\ref{linearproblem}).
Equate to zero the derivative of the Lagrangian with respect to $\m{X}_{N+1}$
to obtain
\[
\m{0} = \m{TX}_{N+1} - \m{y}_5 - \g{\mu}_{N+1} = 
\m{TX}_{N+1} - \m{y}_5 - \g{\Lambda}_{N+1}.
\]
This is the 5th component of (\ref{linearproblem}).
The derivative of the Lagrangian with respect to $\m{X}_j$,
$1 \le j \le N$,  gives the relation
\[
\sum_{i=1}^N D_{ij} \g{\mu}_i = \m{A}_j\tr (\g{\mu}_j + \omega_j\g{\mu}_{N+1})
+ \omega_j (\m{Q}_j \m{X}_j + \m{S}_j \m{U}_j + \m{y}_{4j}) .
\]
Change variables from $\g{\mu}$ to $\g{\Lambda}$ and substitute
for $D_{ij}$ using (\ref{Ddagger}) to obtain the 4th component of
(\ref{linearproblem}).
Finally, differentiate the Lagrangian with respect to $\m{X}_0$ to obtain
\[
\g{\mu}_{N+1} - \g{\mu}_0 + \m{y}_3 - \sum_{i=1}^N \omega_i \m{y}_{4i}
- \sum_{i=1}^N D_{i0} \g{\mu}_i = \m{0}.
\]
Substitute for the $D_{i0}$ sum using both (\ref{mu0expand}) and
the 4th component of (\ref{linearproblem}) to obtain the 3rd component
of (\ref{linearproblem}).
\end{proof}

We now wish to bound the change in the solution of (\ref{QP})
in terms of the change in $\m{Y}$.
Let $\g{\chi}(\m{Y})$ denote the solution of the state dynamics
(\ref{h99})--(\ref{h100}) associated with $\m{p} = \m{y}_1$,
$\m{q}_0 = \m{y}_0$, and $\m{q}_1 = \m{y}_2$.
In (\ref{QP}) we make the change of variables
$\m{X} = \m{Z} + \g{\chi}(\m{Y})$.
The dynamics of (\ref{QP}) become
\begin{equation}\label{zdynamics}
\sum_{j=1}^N{D}_{ij}{\bf Z}_j
={\bf A}_i{\bf Z}_i+ {\bf B}_i{\bf U}_i, \quad
\m{Z}_0 = \m{0}, \quad
{\bf Z}_{N+1} = \sum_{j=1}^N\omega_j
\left({\bf A}_j{\bf Z}_j+{\bf B}_j{\bf U}_j
\right) .
\end{equation}
Hence, the effect of the variable change is to remove $\m{Y}$ from the
constraints.
After the change of variables, the linear term in the 
objective of (\ref{QP}) reduces to
\begin{eqnarray*}
\hat{\C{L}}(\m{Z}, \m{U}, \m{Y}) &=& \m{y}_5\tr \m{Z}_{N+1}
- \sum_{i=1}^N\omega_i
\left( \m{y}_{4i}\tr \m{Z}_i + \m{y}_{6i}\tr \m{U}_i \right) \\
&& \quad + \m{Z}_{N+1}\tr \m{T} \g{\chi}_{N+1}(\m{Y}) +
\sum_{i=1}^N \omega_i \left[ \m{Z}_i\tr \m{Q}_i \g{\chi}_i(\m{Y})
+ \m{U}_i\tr \m{S}_i\tr \g{\chi}_i(\m{Y}) \right] ,
\end{eqnarray*}
since $\m{Z}_0 = \m{0}$.
Let $(\m{Z}^j, \m{U}^j)$ denote the solution of (\ref{QP})
corresponding to $\m{Y}^j \in \C{Y}$, $j = 1$ and 2.
By \cite[Lem.~4]{DontchevHager93}, the solution change satisfies
the relation
\begin{equation}\label{lemma4}
\C{Q}(\Delta\m{Z}, \Delta\m{U}) \le
|\hat{\C{L}}(\Delta\m{Z},\Delta\m{U}, \Delta \m{Y})|
\end{equation}
where $\Delta\m{Z} = \m{Z}^1 - \m{Z}^2$, $\Delta\m{U} = \m{U}^1 - \m{U}^2$,
and $\Delta \m{Y} = \m{Y}^1 - \m{Y}^2$.

By (A1) we have the lower bound
\begin{equation}\label{lower}
\C{Q}(\Delta\m{Z}, \Delta\m{U}) \ge
\alpha (|\Delta \m{Z}_{N+1}|^2 +
\|\Delta \bar{\m{Z}}\|_\omega^2
+ \|\Delta \m{U}\|_\omega^2),
\end{equation}
where $\Delta \bar{\m{Z}}$ is the subvector of $\Delta \m{Z}$
corresponding to components 1 through $N$.
The Schwarz inequality applied to the linear terms in (\ref{lemma4}) yields
the upper bound
\begin{eqnarray*}
&\left|\hat{\C{L}}(\Delta\m{Z},\Delta\m{U}, \Delta\m{Y})\right|
\le& \\[.08in]
& c \bigg(
|\Delta \m{Z}_{N+1}|
+ \|\Delta \bar{\m{Z}}\|_\omega
+ \|\Delta \m{U}\|_\omega \bigg)
\bigg( \|\Delta \m{Y}\|_\C{Y} + \|\bar{\g{\chi}}(\Delta\m{Y})\|_\omega
+ |\g{\chi}_{N+1}(\Delta\m{Y})|\bigg).&
\end{eqnarray*}
By (\ref{normbound})
$\|\bar{\g{\chi}}(\Delta\m{Y})\|_\omega \le
\sqrt{2n} \|\bar{\g{\chi}}(\Delta\m{Y})\|_\infty$, and by Lemma~\ref{feasible},
$\|\g{\chi}(\Delta\m{Y})\|_\infty \le c \|\Delta\m{Y}\|_\C{Y}$.
Hence, the upper bound simplifies to
\begin{equation} \label{upper}
\left|\hat{\C{L}}(\Delta\m{Z},\Delta\m{U}, \Delta \m{Y}) \right| \le
c \|\Delta \m{Y}\|_\C{Y}
\bigg( |\Delta \m{Z}_{N+1}|)
+ \|\Delta \bar{\m{Z}}\|_\omega
+ \|\Delta \m{U}\|_\omega \bigg) . 
\end{equation}
Combine (\ref{lemma4})--(\ref{upper}) to obtain the Lipschitz result
\begin{equation}\label{step1}
|\Delta \m{Z}_{N+1}|
+ \|\Delta \bar{\m{Z}}\|_\omega
+ \|\Delta \m{U}\|_\omega  \le
c \|\Delta \m{Y}\|_\C{Y}.
\end{equation}

By (\ref{zdynamics}), we see that $\Delta \m{Z}$ is the solution
of (\ref{h99})--(\ref{h100}) corresponding to
\[
\m{q}_0 = \m{0}, \quad
\m{p}_i = \m{B}_i\Delta\m{U}_i , \quad
\m{q}_1 = \sum_{j=1}^N \omega_j \m{B}_j \Delta \m{U}_j.
\]
By (\ref{step1}), it follows that
\begin{equation}\label{b1}
\|\m{B}\Delta\m{U}\|_\omega \le c \|\Delta \m{U}\|_\omega \le
c \|\Delta \m{Y}\|_\C{Y},
\end{equation}
where $\m{B}$ is the block diagonal matrix with $i$-th diagonal block
$\m{B}_i$.
Moreover, by the Schwarz inequality and (\ref{b1}), we have
\begin{equation} \label{b2}
\left| \sum_{j=1}^N \omega_j \m{B}_j \Delta \m{U}_j \right| \le
\left( \sum_{j=1}^N \omega_j \right)^{1/2} \left[
\sum_{j=1}^N \omega_j |\m{B}_j\Delta \m{U}_j|^2 \right]^{1/2}
\le c \|\Delta \m{Y}\|_\C{Y}.
\end{equation}
Hence, this choice for $\m{q}_0$, $\m{q}_1$, and $\m{p}$ together with
Lemma~\ref{feasible} and the bounds (\ref{b1}) and (\ref{b2}) imply that
$\|\Delta \m{Z}\|_\infty \le c \|\Delta \m{Y}\|_\C{Y}$.
Since $\Delta \m{X} = \Delta \m{Z} + \g{\chi}(\Delta \m{Y})$ where
$\|\g{\chi}(\Delta\m{Y})\|_\infty \le c \|\Delta\m{Y}\|_\C{Y}$
by Lemma~\ref{feasible}, we conclude that
\begin{equation}\label{bound1}
\|\Delta \m{X}\|_\infty \le c \|\Delta \m{Y}\|_\C{Y}.
\end{equation}

Now consider the costate dynamics (\ref{h199})--(\ref{h200}) with
\begin{eqnarray*}
\m{q}_0 &=&
-\left( \Delta \m{y}_3 +
\sum_{j=1}^N \omega_j [\m{Q}_j\Delta \m{X}_j + \m{S}_j \Delta \m{U}_j]
\right), \\
\m{p}_i &=& -\left( \Delta \m{y}_{4i}
+ \m{Q}_i \Delta \m{X}_i + \m{S}_i \Delta \m{U}_i
\right), \\
\m{q}_1 &=& -\Delta \m{y}_5 + \m{T}\Delta \m{X}_{N+1} .
\end{eqnarray*}
By (\ref{normbound}) and (\ref{bound1}), we have
\begin{equation}\label{b3}
\|\m{Q}\Delta\bar{\m{X}}\|_\omega \le c \|\Delta \bar{\m{X}}\|_\omega \le
c \|\Delta \bar{\m{X}}\|_\infty \le c \|\Delta \m{Y}\|_\C{Y},
\end{equation}
where $\m{Q}$ is the block diagonal matrix
with $i$-th diagonal block $\m{Q}_i$.
The $\m{S}_i \Delta \m{U}_i$ term associated with $\m{p}_i$ can be analyzed
as in (\ref{b1}) and the $\m{S}_j \Delta \m{U}_j$ terms in $\m{q}_0$ can
be analyzed as in (\ref{b2}).
Analogous to the state dynamics, it follows from Lemma~\ref{cofeasible} that
\begin{equation}\label{bound2}
\|\Delta \g{\Lambda}\|_\infty \le c \|\Delta \m{Y}\|_\C{Y} .
\end{equation}

Let $[\m{X}(\m{Y}), \m{U}(\m{Y}), \g{\Lambda}(\m{Y})]$ denote the
solution of (\ref{linearproblem}) for given $\m{Y} \in \C{Y}$.
From the last component of the inclusion (\ref{linearproblem}) and
for any $i$ between 1 and $N$, we have
\[
\left[ \m{y}_6 - \m{S}_i\tr\m{X}_i(\m{Y}) - \m{R}_i \m{U}_i(\m{Y})
- \m{B}_i\tr \g{\Lambda}_i(\m{Y}) \right]\tr
(\m{V}-\m{U}_i(\m{Y})) \le 0 \quad \mbox{for all } \m{V} \in \C{U} .
\]
We add the inequality corresponding to
$\m{Y} = \m{Y}^1$ and $\m{V} = \m{U}_i(\m{Y}^2)$ to the inequality
corresponding to
$\m{Y} = \m{Y}^2$ and $\m{V} = \m{U}_i(\m{Y}^1)$ to obtain the inequality
\[
\Delta \m{U}_i\tr\m{R}_i \Delta \m{U}_i \le
\left[ -\Delta \m{y}_6 + \m{S}\tr \Delta \m{X}_i
+ \m{B}_i\tr \Delta \g{\Lambda}_i \right]\tr \Delta \m{U}_i .
\]
By (A1) and the Schwarz inequality, it follows that
\[
\|\Delta \m{U}_i\|_\infty \le |\Delta \m{U}_i| \le
c (|\Delta \m{y}_6| + | \Delta \m{X}_i| + |\Delta \g{\Lambda}_i|) .
\]
We utilize the previously established bounds (\ref{bound1}) and (\ref{bound2})
to obtain $\|\Delta \m{U}\|_\infty \le \|\Delta \m{Y}\|_\C{Y}$.
The following lemma summarizes these observations.
\smallskip

\begin{lemma}\label{inf-bounds}
If {\rm (A1)}--{\rm (A2)} and {\rm (P1)}--{\rm (P2)} hold,
then there exists a constant $c$, independent of $N$,
such that the change
$(\Delta \m{X}, \Delta \m{U}, \Delta \g{\Lambda})$ in the solution of
$(\ref{linearproblem})$ corresponding to a change $\Delta \m{Y}$
in $\m{Y} \in \C{Y}$ satisfies
\[
\max \left\{ \|\Delta \m{X}\|_\infty, \|\Delta \m{U}\|_\infty,
\|\Delta\g{\Lambda}\|_\infty \right\} \le c \|\Delta \m{Y}\|_\C{Y} .
\]
\end{lemma}

Theorem~\ref{maintheorem} follows from 
Lemma~\ref{inf-bounds} and Proposition~\ref{abstractProp};
the proof is a small modification of the analysis in
\cite[Thm.~2.1]{HagerHouRao16c}.
The Lipschitz constant $\mu$ of Proposition~\ref{abstractProp} is
the constant $c$ of Lemma~\ref{inf-bounds}.
Choose $\varepsilon$ small enough that $\varepsilon\mu < 1$.
When we compute the difference
$\nabla \C{T}(\m{X}, \m{U}, \g{\Lambda}) -
\nabla \C{T}(\m{X}^*, \m{U}^*, \g{\Lambda}^*)$ for
$(\m{X}, \m{U}, \g{\Lambda})$ near $(\m{X}^*, \m{U}^*, \g{\Lambda}^*)$,
the $\m{D}$ and $\m{D}^\dagger$ constant terms cancel, and we are
left with terms involving the difference of
derivatives of $\m{f}$ or $C$ up to second order at nearby points.
By the smoothness assumption,
these second derivatives are uniformly continuous on
the closure of $\C{O}$ and on a ball around $\m{x}^*(1)$.
Utilizing (\ref{normbound}), it follows that for
$r$ sufficiently small,
\[
\
\|\nabla\mathcal{T}(\m{X}, \m{U}, \g{\Lambda})-
\nabla\mathcal{T}(\m{X}^*, \m{U}^*, \g{\Lambda}^*)\|_\C{Y}
\leq \varepsilon
\]
whenever
\begin{equation}\label{r-bound}
\max\{\|\m{X} -\m{X}^*\|_\infty, \|\m{U} -\m{U}^*\|_\infty,
\|\g{\Lambda} - \g{\Lambda}^*\|_\infty\} \le r.
\end{equation}
Since the smoothness $\eta \ge 2$ in Theorem~\ref{maintheorem}, let
us choose $\eta = 2$ in Lemma~\ref{residual_lemma} and then take $\bar{N}$
large enough that
$\left\|\mathcal{T}\left(\m{X}^*, \m{U}^*, \g{\Lambda}^*\right)\right\|_\C{Y}
\leq(1-\mu\varepsilon)r/\mu$ for all $N \ge \bar{N}$.
Hence, by Proposition~\ref{abstractProp}, there exists a solution to
$\C{T}(\m{X},\m{U}, \g{\Lambda}) \in \C{F}(\m{U})$ satisfying (\ref{r-bound}).
Moreover, by (\ref{abs}) and (\ref{resbound}),
the estimate (\ref{maineq}) holds.
We can use exactly the same argument given in \cite{HagerHouRao16c} to show
that this solution to the first-order condition
$\C{T}(\m{X},\m{U}, \g{\Lambda}) \in \C{F}(\m{U})$ is a local minimizer
of (\ref{D}) or equivalently, of (\ref{nlp}).

\section{Numerical experiments}
\label{numerical}
We consider the problem from \cite{Hager84b} given by
\begin{eqnarray}
&\mbox{minimize}&\quad \frac{1}{2}\int_{0}^{1}[x^2(t)+u^2(t)]\;dt\nonumber\\
&\mbox{subject to}& \quad \dot{x}(t)=u(t),\quad u(t)\leq 1, \quad t \in \Omega,
\quad x(0)=\frac{1+3e}{2(1-e)}. \label{example}
\end{eqnarray}
The optimal state and control are
\[
\begin{array}{lll}
{0\leq t\leq \frac{1}{2}}:
& x^*(t)=\displaystyle{t+\frac{1+3e}{2(1-e)},}
& u^*(t)=1,\\[.10in]
{\frac{1}{2} \leq t\leq 1}:
& x^*(t)=\displaystyle{\frac{e^t + e^{2-t}}{\sqrt{e}(1-e)},}
& \displaystyle{u^*(t)=\frac{e^t - e^{2-t}}{\sqrt{e}(1-e)} .}
\end{array}
\]
The associated costate is the integral of the state from $t$ to 1.
Since the objective of the test problem is quadratic and the constraints
are linear equalities and inequalities, the discrete problem
(\ref{nlp}) is a quadratic programming problem, which we solved
using MATLAB's routine {\sc quadprog}.
In Figure~\ref{problem1}, we plot in base 10 the logarithm of the sup-norm
error in the state, control, and costate versus the logarithm of the
degree of the polynomial in the discrete problem.
Since the optimal state has a discontinuity in its second derivative at
$t = 1/2$, $x^*$ lies in $H^2([0, 1])$ as well as in the fractional
Sobolev space $H^{2.5-\epsilon}([0, 1])$ for any $\epsilon > 0$.
Theorem~\ref{maintheorem} implies that the error is $O(N^{\epsilon-1})$.
On the other hand, the observed convergence rate in Figure~\ref{problem1}
is $O(N^{-2})$,
so the error bound given in Theorem~\ref{maintheorem} is not tight,
at least for this particular test problem.
\begin{figure}
\begin{center}
\includegraphics[width=9.0cm]{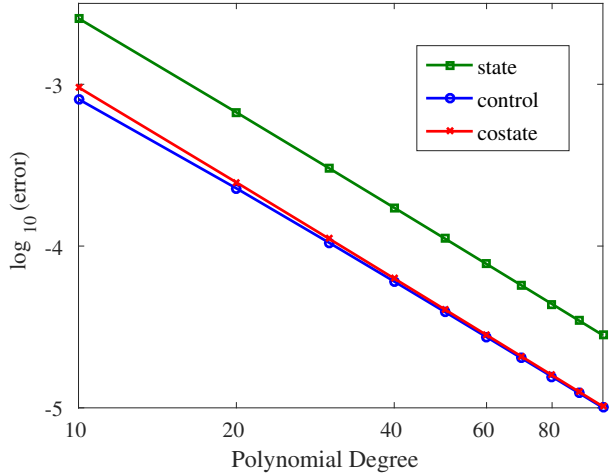}
\end{center}
\caption{Logarithm of sup-norm error in state, control, and costate
versus polynomial degree.}
\label{problem1}
\end{figure}

\section{Conclusions}
An estimate is obtained for the sup-norm error in an approximation
to a control constrained variational problem where the
state is approximated by a polynomials of degree $N$ and the dynamics
is enforced at the Gauss quadrature points.
The error is bounded by $(c/N)^{p-3/2}$ times the
$\C{H}^p$ norms of the state and costate, where $p$ is the minimum
of $N+1$ and the smoothness $\eta$; it is assumed that $\eta \ge 2$.
In \cite{HagerHouRao16c}, an unconstrained control problem was considered and
the corresponding bound was $(c/N)^{p-3}$ with $\eta \ge 4$.
The new work advances the convergence theory by
significantly improving the exponent in the convergence rate, by relaxing the
smoothness requirement, and by including control constraints.
When control constraints are present, $\eta$ is often at most 2,
so the relaxation in the smoothness condition is needed to treat control
constrained problems.
When control constraints are introduced, the first-order optimality conditions
lead to a variational inequality, and
the analysis centers on the stability of the linearized variational problem
under perturbations.
The improvements in the convergence theory were achieved by analyzing
the effect of perturbations
in an $L^2 (\Omega)$ setting rather in $L^\infty (\Omega)$,
and by analyzing interpolation errors in the
Sobolev space $\C{H}^p(\Omega)$ rather than in $L^\infty (\Omega)$.
A numerical example indicates that further tightening of the convergence
theory may be possible.

\section{Appendix 1: Proof of (P1)}
\label{appendix1}
Let $p \in \C{P}_N$ be any polynomial for which $p(-1) = 0$ and
let $\m{p}$ and $\dot{\m{p}} \in \mathbb{R}^N$ denote the vectors
with components $p(\tau_i)$ and $\dot{p}(\tau_i)$ respectively,
$1 \le i \le N$.
Since $p(-1) = 0$, the differentiation matrix $\m{D}$ satisfies
$\m{D}_{1:N} \m{p} = \dot{\m{p}}$, or equivalently,
$\m{D}_{1:N}^{-1} \dot{\m{p}} = \m{p}$.
Let $\m{r}\tr$ denote the $j$-th row of $\m{D}_{1:N}^{-1}$ for any $j$
between 1 and $N$, and let $\dot{\m{p}}$
have components $+1$ or $-1$ where the sign is chosen so that
\[
\m{r}\tr\dot{\m{p}} = \sum_{i=1}^N |r_i|.
\]
Due to the identity $\m{D}_{1:N}^{-1} \dot{\m{p}} = \m{p}$, we conclude that
\[
\sum_{i=1}^N |r_i| = p(\tau_j).
\]
Hence, (P1) holds if $|p(\tau_j)| \le 2$ whenever $p \in \C{P}_N$ is a
polynomial that satisfies $p(-1) = 0$ and $|\dot{p}(\tau_i)| \le 1$
for all $1 \le i \le N$.
We will prove the following stronger result:
\smallskip

\begin{proposition}\label{P1bound}
For any $p \in \C{P}_N$ with
$p(-1) = 0$ and $|\dot{p}(\tau_i)| \le 1$ for all $1 \le i \le N$,
we have $|p(\tau)| \le 2$ for all $\tau \in [-1, 1]$.
\end{proposition}
\smallskip

\begin{proof}
Let $l_i$, $1 \le i \le N$, be the Lagrange interpolating polynomials
defined by
\[
l_i (\tau) = \prod^{N}_{\substack{j=1\\j\neq i}}
\frac{\tau-\tau_j}{\tau_i-\tau_j}.
\]
Let $p \in \C{P}_N$ be any polynomial with
$p(-1) = 0$ and $|\dot{p}(\tau_i)| \le 1$ for all $1 \le i \le N$.
Since $\dot{p} \in \C{P}_{N-1}$, we can write
\[
\dot{p}(\tau) = \sum_{i=1}^N \dot{p}(\tau_i) l_i (\tau).
\]
Since $|\dot{p}(\tau_i)| \le 1$, it follows that
\begin{equation}\label{p1}
|p(t)| =\left| \int_{-1}^t \dot{p}(\tau) \; d\tau \right| =
\left| \sum_{i=1}^N \dot{p}(\tau_i) \int_{-1}^t l_i(\tau) \; d\tau \right|
\le \sum_{i=1}^N \left| \int_{-1}^t l_i (\tau) \; d\tau \right| .
\end{equation}

Let $q \in \C{P}_{N-1}$ be defined by
\[
q(\tau) = \sum_{i=1}^N a_i l_i (\tau) \quad \mbox{where} \quad
a_i = \left\{ \begin{array}{rl}
1 & \mbox{if } \displaystyle{\int_{-1}^t} l_i (\tau) \; d\tau > 0, \\
-1 & \mbox{otherwise} .
\end{array} \right.
\]
Hence, we have
\begin{equation}\label{p2}
\sum_{i=1}^N \left| \int_{-1}^t l_i (\tau) \; d\tau \right| =
\sum_{i=1}^N a_i \int_{-1}^t l_i (\tau) \; d\tau =
\int_{-1}^t q(\tau) \; d\tau.
\end{equation}
Since $|q(\tau_i)| = |a_i| = 1$ for each $i$, it follows that
$q^2(\tau) -1$ vanishes at $\tau = \tau_i$, $1 \le i \le N$.
Since $q^2 \in \C{P}_{2N-2}$, we have the factorization
\begin{equation}\label{q2}
q^2(\tau) - 1 =
r(\tau) P_N(\tau),
\end{equation}
where $r \in \C{P}_{N-2}$ and $P_N$ is the Legendre polynomial of degree $N$.
Since $P_N$ is orthogonal to polynomials of degree at most $N-1$, the integral
of (\ref{q2}) yields the identity
\[
\int_{-1}^1 q^2(\tau) \; d\tau = 2.
\]
By the Schwarz inequality,
\[
\int_{-1}^1 |q(\tau)| \; d\tau \le
\left( \int_{-1}^1 d\tau \right)^{1/2}
\left( \int_{-1}^1 q^2(\tau) \; d \tau \right)^{1/2} = 2.
\]
Combine this with (\ref{p1}) and (\ref{p2}) to obtain
\[
|p(t)| \le
\int_{-1}^1 |q(\tau)| \; d\tau \le 2,
\]
which completes the proof.
\end{proof}

Although this paper has focused on the Gauss abscissa,
Proposition~\ref{P1bound} holds when the Gauss abscissa are
replaced by the Radau abscissa.
\smallskip

\begin{corollary}\label{radau}
If $\tau_i$, $1 \le i \le N$, are the Radau abscissa with $\tau_N = 1$,
then for any $p \in \C{P}_N$ with
$p(-1) = 0$ and $|\dot{p}(\tau_i)| \le 1$ for all $1 \le i \le N$,
we have $|p(\tau)| \le 2$ for all $\tau \in [-1, 1]$.
\end{corollary}
\smallskip

\begin{proof}
Recall that the interior Radau abscissa $\tau_i$, $1 \le i \le N-1$,
are the roots of the Jacobi polynomial $P_{N-1}^{(1,0)}$ associated
with the weight function $1 - \tau$.
The proof of the corollary is identical to the proof of
Proposition~\ref{P1bound} until equation (\ref{q2}), which is replaced by
\begin{equation}\label{q3}
q^2(\tau) - 1 =
r(\tau) P_{N-1}^{(1.0)}(\tau)(\tau_N - \tau),
\end{equation}
where $r \in \C{P}_{N-2}$.
Since $P_{N-1}^{(1,0)}$ is orthogonal to all polynomials in $\C{P}_{N-2}$
with respect to the weight function $1 - \tau = \tau_N - \tau$,
the integral of (\ref{q3}) again yields the identity
\[
\int_{-1}^1 q^2(\tau) \; d\tau = 2.
\]
The remainder of the proof is identical to that of Proposition~\ref{P1bound}.
\end{proof}
\smallskip

\begin{remark}
The polynomial $p (\tau) = 1 + \tau$ satisfies the conditions of
Proposition~$\ref{P1bound}$ and Corollary~$\ref{radau}$, and $p(1) = 2$.
Hence, the upper bound $2$ is tight.
\end{remark}
\smallskip

\begin{remark}
For the Radau abscissa with $\tau_1 = -1$ and $\tau_N < 1$,
the condition $p(-1) = 0$ in the statement of Corollary~$\ref{radau}$ should
be replaced by $p(1) = 0$.
\end{remark}

\section{Appendix 2:
$\C{L}^2$ approximation with a singular weight by Yvon Maday}
\label{appendix2}
In (\ref{quad}) we integrate the error $u - \pi_N u$
in best $\C{H}^1(\Omega)$ approximation
using a singular weight $1/(1-\tau^2)$.
Here we relate this singular integral of the error to the error in the
$\C{H}_0^1(\Omega)$ norm.
\smallskip
\begin{proposition}\label{singular}
If $u \in \C{H}_0^1(\Omega)$, then
\begin{equation}\label{Z}
\|u - \pi_N u\|_0 \le N^{-1} |u - \pi_N u|_{\C{H}^1(\Omega)} , \quad
\mbox{where }
\|v\|_0 = 
\left( \int_\Omega \frac{v^2(\tau)}{1-\tau^2} \; d\tau\right)^{1/2} ,
\end{equation}
and $\pi_N$ is the projection into $\C{P}_N^0$ relative to the
the norm $| \cdot |_{\C{H}^1(\Omega)}$.
\end{proposition}
\smallskip
\begin{proof}
Let $\langle \cdot, \cdot \rangle_1$ denote the standard
$\C{H}_0^1(\Omega)$ inner product defined by
\[
\langle u , v \rangle_1 = \int_{\Omega} u'(\tau) v'(\tau) \; d\tau.
\]
By the Legendre equation,
the polynomials $\psi_k (\tau) := (1-\tau^2)P_k'(\tau)$ are orthogonal
with respect to the $\C{H}_0^1(\Omega)$ inner product and
\[
\langle \psi_k , \psi_k \rangle_1 =
\langle (1-\tau^2)P_k', (1-\tau^2)P_k' \rangle_1 =
k^2(k+1)^2 \langle P_k, P_k \rangle_{\C{L}^2(\Omega)} =
\frac{ 2k^2(k+1)^2}{2k+1} .
\]
Consequently, $\{\psi_k : 1 \le k \le N-1\}$, is an orthogonal basis
for $\C{P}_N^0$, and the orthogonal
projection of $u$ into $\C{P}_N^0$ is given by
\[
\pi_N u = \sum_{i=1}^{N-1} u_k \psi_k, \quad
u_k = \frac{\langle u, \psi_k \rangle_1}{\langle \psi_k , \psi_k \rangle_1}.
\]

Let $\langle \cdot , \cdot \rangle_0$ denote the inner product on
$\C{H}_0^1(\Omega)$ defined by
\[
\langle u , v \rangle_0 = \int_{\Omega} \frac{u(\tau) v(\tau)}{1-\tau^2}
\; d\tau.
\]
By the Schwarz and Hardy inequalities,
$\|u \|_0^2 \le 2 |u|_{\C{H}^1(\Omega)} \|u\|_{\C{L}^2(\Omega)}$.
By the Legendre equation, the $\psi_k$ are also
orthogonal in the $\langle \cdot , \cdot \rangle_0$
inner product and
\begin{eqnarray*}
\langle \psi_k , \psi_k \rangle_0 &=&
\langle (1-\tau^2)P_k', (1-\tau^2)P_k' \rangle_0 =
\langle (1-\tau^2)P_k', P_k' \rangle_{\C{L}^2(\Omega)} \\
&=& k(k+1) \langle P_k, P_k \rangle_{\C{L}^2(\Omega)} = \frac{ 2k(k+1)}{2k+1} .
\end{eqnarray*}
Due to orthogonality, we have
\begin{eqnarray*}
\|u - \pi_N u\|_0^2 &=& \sum_{k\ge N} u_k^2 \langle \psi_k, \psi_k \rangle_0 =
\sum_{k\ge N} \left( \frac{2k(k+1)}{2k+1}\right) u_k^2 , \\
|u - \pi_N u|_{\C{H}^1(\Omega)}^2 &=&
\sum_{k\ge N} u_k^2 \langle \psi_k, \psi_k \rangle_1 =
\sum_{k\ge N} \left( \frac{2k^2(k+1)^2}{2k+1}\right) u_k^2 .
\end{eqnarray*}
Comparing these norms, we see that (\ref{Z}) holds.
\end{proof}

{\bf Acknowledgments.}
We thank the reviewers for their careful reading of the manuscript
and their constructive comments and suggestions.
In particular, one reviewer suggested a better arrangement for
the proof of Lemma~\ref{interp}.
The authors deeply appreciate Yvon Maday's contribution of
Proposition~\ref{singular}, a key step in Lemma~\ref{interp}.

\bibliographystyle{siam}

\begin{thebibliography}{10}

\bibitem{Bellman44}
{\sc R.~Bellman}, {\em A note on an inequality of {E.~Schmidt}}, Bull. Amer.
  Math. Soc., 50 (1944), pp.~734--737.

\bibitem{Benson2}
{\sc D.~A. Benson, G.~T. Huntington, T.~P. Thorvaldsen, and A.~V. Rao}, {\em
  Direct trajectory optimization and costate estimation via an orthogonal
  collocation method}, J. Guid. Control Dyn., 29 (2006), pp.~1435--1440.

\bibitem{BernardiMaday92}
{\sc C.~Bernardi and Y.~Maday}, {\em Polynomial interpolation results in
  {Sobolev} spaces}, J. Comput. Appl. Math.,  (1992), pp.~53--82.

\bibitem{Bonnans10}
{\sc J.~F. Bonnans}, {\em Lipschitz solutions of optimal control problems with
  state constraints of arbitrary order}, Ann. Acad. Rom. Sci. Ser. Math. Appl,
  2 (2010), pp.~78--98.

\bibitem{HagerDontchevPooreYang95}
{\sc A.~Dontchev, W.~W. Hager, A.~Poore, and B.~Yang}, {\em Optimality,
  stability, and convergence in nonlinear control}, Applied Math. and Optim.,
  31 (1995), pp.~297--326.

\bibitem{DontchevHager93}
{\sc A.~L. Dontchev and W.~W. Hager}, {\em Lipschitzian stability in nonlinear
  control and optimization}, {SIAM} J. Control Optim., 31 (1993), pp.~569--603.

\bibitem{DontchevHager98a}
\leavevmode\vrule height 2pt depth -1.6pt width 23pt, {\em A new approach to
  {Lipschitz} continuity in state constrained optimal control}, Systems and
  Control Letters, 35 (1998), pp.~137--143.

\bibitem{DontchevHager97}
\leavevmode\vrule height 2pt depth -1.6pt width 23pt, {\em The {Euler}
  approximation in state constrained optimal control}, Math. Comp., 70 (2001),
  pp.~173--203.

\bibitem{DontchevHagerMalanowski00}
{\sc A.~L. Dontchev, W.~W. Hager, and K.~Malanowski}, {\em Error bounds for
  {Euler} approximation of a state and control constrained optimal control
  problem}, Numer. Funct. Anal. Optim., 21 (2000), pp.~653--682.

\bibitem{DontchevHagerVeliov00}
{\sc A.~L. Dontchev, W.~W. Hager, and V.~M. Veliov}, {\em Second-order
  {Runge}-{Kutta} approximations in constrained optimal control}, {SIAM} J.
  Numer. Anal., 38 (2000), pp.~202--226.

\bibitem{Elnagar1}
{\sc G.~Elnagar, M.~Kazemi, and M.~Razzaghi}, {\em The pseudospectral
  {Legendre} method for discretizing optimal control problems}, IEEE Trans.
  Automat. Control, 40 (1995), pp.~1793--1796.

\bibitem{Elnagar4}
{\sc G.~N. Elnagar and M.~A. Kazemi}, {\em Pseudospectral {Chebyshev} optimal
  control of constrained nonlinear dynamical systems}, Comput. Optim. Appl., 11
  (1998), pp.~195--217.

\bibitem{Elschner93}
{\sc J.~Elschner}, {\em The h-p-version of spline approximation methods for
  {Melin} convolution equations}, J. Integral Equations Appl., 5 (1993),
  pp.~47--73.

\bibitem{Fahroo2}
{\sc F.~Fahroo and I.~M. Ross}, {\em Costate estimation by a {Legendre}
  pseudospectral method}, J. Guid. Control Dyn., 24 (2001), pp.~270--277.

\bibitem{FahrooRoss02}
\leavevmode\vrule height 2pt depth -1.6pt width 23pt, {\em Direct trajectory
  optimization by a {Chebyshev} pseudospectral method}, J. Guid. Control Dyn.,
  25 (2002), pp.~160--166.

\bibitem{Fahroo3}
\leavevmode\vrule height 2pt depth -1.6pt width 23pt, {\em Pseudospectral
  methods for infinite-horizon nonlinear optimal control problems}, J. Guid.
  Control Dyn., 31 (2008), pp.~927--936.

\bibitem{GargHagerRao11a}
{\sc D.~Garg, M.~A. Patterson, C.~L. Darby, C.~Fran\c{c}olin, G.~T. Huntington,
  W.~W. Hager, and A.~V. Rao}, {\em Direct trajectory optimization and costate
  estimation of finite-horizon and infinite-horizon optimal control problems
  using a {Radau} pseudospectral method}, Comput. Optim. Appl., 49 (2011),
  pp.~335--358.

\bibitem{GargHagerRao10a}
{\sc D.~Garg, M.~A. Patterson, W.~W. Hager, A.~V. Rao, D.~A. Benson, and G.~T.
  Huntington}, {\em A unified framework for the numerical solution of optimal
  control problems using pseudospectral methods}, Automatica, 46 (2010),
  pp.~1843--1851.

\bibitem{GongRossKangFahroo08}
{\sc Q.~Gong, I.~M. Ross, W.~Kang, and F.~Fahroo}, {\em Connections between the
  covector mapping theorem and convergence of pseudospectral methods for
  optimal control}, Comput. Optim. Appl., 41 (2008), pp.~307--335.

\bibitem{Hag2}
{\sc W.~W. Hager}, {\em Lipschitz continuity for constrained processes}, {SIAM}
  J. Control Optim., 17 (1979), pp.~321--337.

\bibitem{Hager90}
\leavevmode\vrule height 2pt depth -1.6pt width 23pt, {\em Multiplier methods
  for nonlinear optimal control}, {SIAM} J. Numer. Anal., 27 (1990),
  pp.~1061--1080.

\bibitem{Hager99c}
\leavevmode\vrule height 2pt depth -1.6pt width 23pt, {\em {Runge}-{Kutta}
  methods in optimal control and the transformed adjoint system}, Numer. Math.,
  87 (2000), pp.~247--282.

\bibitem{Hager02b}
\leavevmode\vrule height 2pt depth -1.6pt width 23pt, {\em Numerical analysis
  in optimal control}, in International Series of Numerical Mathematics, K.-H.
  Hoffmann, I.~Lasiecka, G.~Leugering, J.~Sprekels, and F.~Tr\"{o}ltzsch, eds.,
  vol.~139, Basel/Switzerland, 2001, Birkhauser Verlag, pp.~83--93.

\bibitem{HagerHouRao16a}
{\sc W.~W. Hager, H.~Hou, S.~Mohapatra, and A.~V. Rao}, {\em Convergence rate
  for an {\it hp}-collocation method applied to constrained optimal control},
  (2016, arXiv: 1605.02121).

\bibitem{HagerHouRao15c}
{\sc W.~W. Hager, H.~Hou, and A.~V. Rao}, {\em Convergence rate for a {Radau}
  collocation method applied to unconstrained optimal control},  (2015, arXiv:
  1508.03783).

\bibitem{HagerHouRao16c}
\leavevmode\vrule height 2pt depth -1.6pt width 23pt, {\em Convergence rate for
  a {Gauss} collocation method applied to unconstrained optimal control}, J.
  Optim. Theory Appl., 169 (2016), pp.~801--824.

\bibitem{Hager84b}
{\sc W.~W. Hager and G.~Ianculescu}, {\em Dual approximations in optimal
  control}, {SIAM} J. Control Optim., 22 (1984), pp.~423--465.

\bibitem{Hermant09}
{\sc A.~Hermant}, {\em Stability analysis of optimal control problems with a
  second-order state constraint}, {SIAM} J. Optim., 22 (2009), pp.~104--129.

\bibitem{HJ12}
{\sc R.~A. Horn and C.~R. Johnson}, {\em Matrix Analysis}, Cambridge University
  Press, Cambridge, 2013.

\bibitem{Kameswaran1}
{\sc S.~Kameswaran and L.~T. Biegler}, {\em Convergence rates for direct
  transcription of optimal control problems using collocation at {Radau}
  points}, Comput. Optim. Appl., 41 (2008), pp.~81--126.

\bibitem{kang08}
{\sc W.~Kang}, {\em The rate of convergence for a pseudospectral optimal
  control method}, in Proceeding of the 47th IEEE Conference on Decision and
  Control, IEEE, 2008, pp.~521--527.

\bibitem{kang10}
\leavevmode\vrule height 2pt depth -1.6pt width 23pt, {\em Rate of convergence
  for the {Legendre} pseudospectral optimal control of feedback linearizable
  systems}, J. Control Theory Appl., 8 (2010), pp.~391--405.

\bibitem{LiuHagerRao15}
{\sc F.~Liu, W.~W. Hager, and A.~V. Rao}, {\em Adaptive mesh refinement method
  for optimal control using nonsmoothness detection and mesh size reduction},
  J. Franklin Inst., 352 (2015), pp.~4081--4106.

\bibitem{Markov1916}
{\sc V.~A. Markov}, {\em \"{U}ber {Polynome}, die in einem gegebenen
  {I}ntervalle m\"{o}glichst wenig von {N}ull abweichen}, Math. Ann., 77
  (1916), pp.~185--191.

\bibitem{NocedalWright2006}
{\sc J.~Nocedal and S.~J. Wright}, {\em Numerical Optimization}, Springer, New
  York, 2nd~ed., 2006.

\bibitem{PattersonHagerRao14}
{\sc M.~A. Patterson, W.~W. Hager, and A.~V. Rao}, {\em A $ph$ mesh refinement
  method for optimal control}, Optim. Control Appl. Meth., 36 (2015),
  pp.~398--421.

\bibitem{Reddien79}
{\sc G.~W. Reddien}, {\em Collocation at {Gauss} points as a discretization in
  optimal control}, {SIAM} J. Control Optim., 17 (1979), pp.~298--306.

\bibitem{StoerBulirsch02}
{\sc J.~Stoer and R.~Bulirsch}, {\em Introduction to Numerical Analysis},
  Springer-Verlag, 3rd~ed., 2002.

\bibitem{Williams1}
{\sc P.~Williams}, {\em Jacobi pseudospectral method for solving optimal
  control problems}, J. Guid. Control Dyn., 27 (2004), pp.~293--297.

\end{thebibliography}

\end{document}